\documentclass{article}

\usepackage{color, amsmath, amsfonts, amsthm, amssymb, amsrefs,physics,mathtools} 
\newtheorem{theorem}{Theorem}[section]

\newtheorem{lemma}[theorem]{Lemma}
\newtheorem{proposition}[theorem]{Proposition}
\newcommand{\Rd}{\mathbb{R}^d}
\newcommand{\boldc}{\mathbf{c}}
\newcommand{\Reals}{\mathbb{R}}
\newcommand{\Rn}{\mathbb{R}^n}

\title{Convergence to the complex balanced equilibrium for some chemical reaction-diffusion systems with boundary equilibria}

\author{
Gheorghe Craciun
\thanks{Department of Mathematics 
and Department of Biomolecular Chemistry, 
University of Wisconsin-Madison, 
{\tt craciun@math.wisc.edu}}
\and
Jiaxin Jin
\thanks{Department of Mathematics,
University of Wisconsin-Madison
{\tt jjin43@wisc.edu}}
\and
Casian Pantea
\thanks{Department of Mathematics, 
West Virginia University,  
{\tt cpantea@math.wvu.edu}}
\and
Adrian Tudorascu 
\thanks{Department of Mathematics,
West Virginia University, 
{\tt adriant@math.wvu.edu}}
}

\date{}

\begin{document}

\maketitle

%

\begin{abstract}
\noindent In this paper we study the rate of convergence to the complex balanced equilibrium for some chemical reaction-diffusion systems with boundary equilibria. We first analyze a three-species system with boundary equilibria in some stoichiometric classes, and whose right hand side is  bounded above by a quadratic nonlinearity in the positive orthant. We prove similar results on the convergence to the positive equilibrium for a fairly general two-species reversible reaction-diffusion network with boundary equilibria.

\end{abstract}

\section{Introduction} 
The dynamical behavior of spatially homogeneous mass-action reaction systems has been the focus of much research over the last fifty years. These ODE systems are usually high-dimensional, non-linear, and depend on a large number of parameters, which makes them generally difficult to study. However, a fertile theory started fifty years ago with work of Horn, Jackson and Feinberg \cite{Horn.1972, Horn.1974, Feinberg.1972} has been successful in addressing questions of existence and stability of positive equilibria, and persistence (nonextinction) of variables. Their work shows that, surprisingly, the large class of {\em complex balanced} mass-action systems have unique positive equilibria and admit a global Lyapunov function, which makes them locally asymptotically stable independently of reaction rate constant values. This robustness is relevant in applications, where exact values of system parameters are typically unknown. Moreover, Horn conjectured that the unique equilibria are in fact globally asymptotically stable \cite{Horn.1974}, a question known as the {\em Global Attractor Conjecture}. The conjecture stayed open until recent years, when new work fueled in part by advances in systems biology led to a series of partial results. It was shown that  trajectories of complex balanced systems either converge to the positive equilibrium or go to boundary equilibria \cite{Siegel.2000, Sontag.2001}, establishing that persistence implies global stability. A series of subsequent papers showed persistence for complex balanced systems in two variables and other classes of systems, and proved the Global Attractor Conjecture in two and three variables \cite{Anderson.2008aa, Craciun.2009aa, Anderson.2010aa, Anderson.2011aa, Pantea.2012, Craciun.2013aa, Gopalkrishnan.2013, Gopalkrishnan.2013}. This work led to a very recent proof of the Global Attractor Conjecture in full generality by Craciun \cite{Craciun.gac}.   

Much less is known about the corresponding reaction-diffusion models, although a number of recent papers have focused on extending the results above in the PDE setting. A promising venue for relating the PDE and ODE models is by way of space discretization (the method of lines). As proof of concept, the network $A+B\rightleftharpoons C$ was considered in \cite{FatmaCA} where it was shown that solutions of the discretized system converge to the solution of the PDE system as the space discretization grows finer. Solutions of the reaction-diffusion system $A+B\rightleftharpoons C$ have been shown to approach a positive spatially homogeneous distribution \cite{Rothe.1984} via semigroup theory. Newer work uses entropy techniques to prove global asymptotic stability for other systems, including dimerization networks  $2A\rightleftharpoons B$ \cite{DF06} and monomolecular networks \cite{Fellner.2015aa}. 

For general complex balanced systems it was shown that under the assumption of equal diffusion constants, their $\omega$-limit set consists of constant functions corresponding to equilibria of the space-homogeneous ODE system and that, moreover, the unique positive equilibrium is asymptotically stable. This is the analogous of the ``persistence implies global stability" result from the ODE setting, albeit in the case of equal diffusion constants. Recent results by Desvillettes, Fellner and collaborators \cite{DF14, DFT2} removed the requirement of equal diffusion constants, and showed that in the absence of boundary equilibria, the positive equilibrium of general complex balanced reaction-diffusion systems attracts all solutions with positive initial data. These papers also considered special cases of networks with boundary equilibria, where a more detailed analysis showed that positive solutions remain globally asymptotically stable. However, the general case of systems with boundary equilibria remains open, and the analysis of such systems is on a case-by-case basis.   

Our paper studies two cases of complex balanced reaction networks with boundary equilibria, and shows that under mild boundedness conditions on the initial data, solutions converge asymptotically to the unique positive equilibria. Namely, we consider the three-species system $A+nB{\rightleftharpoons}B+C$ (Theorem \ref{conv-theorem1}), and the two-species system 
$m_1A+n_1B{\rightleftharpoons}m_2A+n_2B$ (Theorem \ref{conv-theorem2}). 

In the remainder of this introductory section we set up terminology and notation, we discuss some of the techniques used here and in previous work, and we state our main theorems. Sections \ref{3x3-Syst} and \ref{2x2-Syst} contain the proofs of the results for the three-species and two-species systems. We conclude with a few remarks and open problems (Section \ref{sec:remarks}) and with an appendix collecting a few technical results needed in the paper.

\subsection{Terminology and previous results}
Let us consider $0<T\leq\infty$ and the semilinear parabolic system 
\begin{equation}\label{general-system}
\boldc_t-\mathcal{D}\Delta \boldc=R(\boldc)\mbox{ in }\Omega\times(0,T),\ \boldc(\cdot,0)=\boldc_0\mbox{ in }\Omega,
\end{equation}
where $\boldc:\Omega\times[0,T)\rightarrow\Rn$ is the vector of concentrations at spatial position $x\in \Omega$  (an open subset of $\Rd$) and time $t\in[0,\infty)$, $\mathcal{D}$ is a positive definite, diagonal $n\times n$ matrix, and $R:\Rn\rightarrow \Rn$ is a vector field whose components are polynomials (determined by the chemical reactions under consideration).  This system can be linear and ``trivial'' (at least in the sense that ``enough'' of its equations decouple), such as 
\begin{equation}\nonumber
\begin{array}{l}
\displaystyle
a_t-d_a\Delta a= -ka, \\
\displaystyle
b_t-d_b\Delta b=ka\mbox{ in }\Omega\times(0,T),\\
\displaystyle
a(\cdot,0)=a_0,\ b(\cdot,0)=b_0\mbox{ in }\Omega,
\end{array}
\end{equation}
(which corresponds to the reaction $A\displaystyle\rightarrow B$ with reaction rate $k>0$), linear and nontrivial (weakly coupled) such as 
\begin{equation}\nonumber
\begin{array}{l}
\displaystyle
a_t-d_a\Delta a= -k_1a+k_2b, \\
\displaystyle
b_t-d_b\Delta b=k_1a-k_2b\mbox{ in }\Omega\times(0,T),\\
\displaystyle
a(\cdot,0)=a_0,\ b(\cdot,0)=b_0\mbox{ in }\Omega,
\end{array}
\end{equation}
(which corresponds to 
$A  \overset{k_{1}}{\underset{k_2}{\rightleftharpoons}} B$). However, as soon as a reaction includes two or more reactants, the system becomes nonlinear in the zero order terms (semilinear). For example, the single reaction $A+B\overset{k}{\rightarrow} C$ yields  
\begin{equation}\nonumber
\begin{array}{l}
\displaystyle
a_t-d_a\Delta a= -kab, \\
\displaystyle
b_t-d_b\Delta b=-kab,\\
\displaystyle
c_t-d_c\Delta c=kab\mbox{ in }\Omega\times(0,T)\\
\displaystyle
a(\cdot,0)=a_0,\ b(\cdot,0)=b_0,\ c(\cdot,0)=c_0\mbox{ in }\Omega.
\end{array}
\end{equation}
We use this last system to illustrate some terminology and notation. Here $A,\ B,$ and $C$ are the three {\em species} of the network, and $A+B$ and $C$ are its {\em complexes}. In general, complexes are formal linear combinations of species with non-negative integer coefficients, and sit on both sides of a reaction arrow. It is useful to think of complexes as vectors in a natural way, for example $A+B$ corresponds to $y=(1,1,0)$, and $C$ to $y'=(0,0,1)$. 
The concentrations of $A,\ B,\ C$ are non-negative functions of time and space and are collected in the {\em concentration vector} $\boldc=(a,b,c)$. 
The {\em reaction rate} of a reaction is given by mass-action, and is proportional to the concentration of each reactant species. This way, the reaction $A+B\overset{k}{\rightarrow} C$ has rate $kab$. The {\em reaction rate constant} $k$ is a reaction-specific positive number. 
In general, the rate of a the reaction $y\overset{k}{\to} y'$ is given by 
$$k\boldc^{y}=k\displaystyle\prod_{i=1}^nc_i^{y_i},$$  
where $n$ is the number of species, and complexes $y$ and $y'$ are viewed as vectors, as illustrated above. 
The reaction rate $kab$ enters with negative sign in the equations for $a_t$ and $b_t$ ($A$ and $B$ are being consumed in the reaction), and with positive sign in the equation for $c_t$ ($C$ is being produced). The aggregate contribution of all reaction rates are collected in the vector $R(\boldc)=(-kab,-kab,kab)$. In general, this is given by        
$$R(\boldc):=\sum_{y\to y'} k_{y\to y'}\boldc^{y}(y'-y),$$
where $k_{y \to y'}$ is the rate constant of $y\to y'$ and the summation is over all reactions $y\to y'$ in the network.
Finally, $\mathcal{D}=\mathrm{diag}\{d_a,d_b,d_c\}\in M_{3\times 3}(\Reals)$ denotes the diagonal matrix of diffusion constants. 

In the previous example the first two equations have the benefit of being decoupled, but that feature is lost as soon as we allow for reversibility; indeed, corresponding to $A  + B\overset{k_{1}}{\underset{k_2}{\rightleftharpoons}} C$ we have 
\begin{equation}\label{nonlinear-eq1}
\begin{array}{l}
\displaystyle
a_t-d_a\Delta a= -k_1ab+k_2c, \\
\displaystyle
b_t-d_b\Delta b=-k_1ab+k_2c,\\
\displaystyle
c_t-d_c\Delta c=k_1ab-k_2c\mbox{ in }\Omega\times(0,T)\\
\displaystyle
a(\cdot,0)=a_0,\ b(\cdot,0)=b_0,\ c(\cdot,0)=c_0\mbox{ in }\Omega.
\end{array}
\end{equation}

When it comes to basic questions on the existence, uniqueness, smoothness and non-negativity of solutions (if the initial data components are nonnegative), for linear systems the answers are provided in the (by now, classical) literature (see, e.g., \cite{Protter-Weinberger}).  However, complexity adds quickly as more reactions and/or more reactants enter the system. There is no general result in the literature that guarantees long time existence, uniqueness and non-negativity of solutions, let alone smoothness and other, more delicate properties such as comparison principles and convergence to equilibrium. In this paper we discuss chemical reaction diffusion systems which have a specific structure relative to a {\it positive equilibrium}, i.e. a steady state solution with all positive components. 

In general, we say that an equilibrium point $c_0$ is a {\em complex balanced equilibrium} if for all complexes $\bar y$ we have  
$$\sum_{\bar y \to y} k_{\bar y \to y}c_0^{\bar y} = \sum_{y \to \bar y} k_{y \to \bar y}c_0^{y}$$ 
i.e., the total chemical flux that exits the complex $\bar y$ equals the total chemical flux that enters the complex $\bar y$ (for any choice of $\bar y$).  A chemical system is called a {\em complex balanced system} if it admits a positive complex balanced equilibrium. 
Due to the particular polynomial nature of $R(\boldc)$ in this case, the steady states are all constant vectors (i.e. independent of location, as well as of time). If at least one component of a steady state is null, then the corresponding state is said to be a {\it boundary equilibrium}. 
All systems arising from complex balanced CRDSs admit a ``canonical'' Lyapunov functional of the relative Boltzmann entropy type. Its general form (again, see, e.g., \cite{DFT2}), this logarithmic free relative energy functional reads
\begin{equation}\nonumber
E(t):=\sum_{i=1}^n\int_\Omega\bigg[c_i(x,t)\log\frac{c_i(x,t)}{c_{i,\infty}}-c_i(x,t)+c_{i,\infty}\bigg]dx,
\end{equation}
where $\boldc_\infty:=(c_{1,\infty},...,c_{n,\infty})$ is the constant vector denoting the positive complex balanced equilibrium. The entropy dissipation functional is computed by differentiating $E$ along trajectories; that is, once all the time derivatives of concentrations are replaced by their equation specific expressions and the Neumann BC are used to integrate by parts wherever the Laplacian appears, one gets
\begin{equation}\nonumber
D(t):=\sum_{i=1}^n d_i\int_\Omega\frac{|\nabla c_i(x,t)|^2}{c_i(x,t)}dx+\sum_{r=1}^{\rho}k_rc_\infty^{y_r}\int_\Omega\Phi\bigg(\frac{c^{y_r}}{c_\infty^{y_r}};\frac{c^{y'_r}}{c_\infty^{y'_r}}\bigg)dx,
\end{equation}
where $\rho$ is the number of reactions and $\Phi(x,y):=x\log(x/y)-x+y$. Of course, one gets exponential decay to zero for $E$ if one can prove that there exists a positive constant $\alpha$ such that 
\begin{equation}\label{EEDI}
D(t)\geq\alpha E(t)\mbox{ for all }t\geq 0.
\end{equation} 
Naturally, $E(t)$ should not only be identically zero when $\boldc(t)=\boldc_\infty$, but it should also be bounded below by some increasing function of the distance (from some norm) between $\boldc(t)$ and $\boldc_\infty$.

For complex balanced systems, in the spatially isotropic case ($\mathcal{D}=0$, so the PDE's are reduced to ODE's) there is some recent work by Craciun~\cite{Craciun.gac}, which  answers in the affirmative a long standing conjecture on the convergence to the positive equilibrium in each stoichiometric class, called the Global Attractor Conjecture. This conjecture states that regardless of the existence of boundary equilibria, trajectories starting in the positive orthant converge to the unique positive equilibrium in the corresponding stoichiometric class.  In the PDE case, the most general result concerns the case where there are no boundary equilibria. Very recently, Desvillettes, Fellner and Tang \cite{DFT2} showed that, contingent on the existence of suitable solutions (essentially, solutions that may not be classical but they are {\it renormalized} and do satisfy a weak entropy entropy-dissipation law), one obtains exponentially fast convergence to the equilibrium which lies in the same stoichiometric class as the initial data, which is merely assumed nonnegative and integrable over some bounded, $C^2$ domain in $\mathbb{R}^d$.  This is also a remarkably general result in the sense that the initial concentrations are only assumed to lie in $L^1(\Omega)$. This improvement (over the previous works, where $L^\infty$-bounds were imposed on the initial data) is achieved via the use of the Log-Sobolev inequality (see, e.g., \cite{DFT2}) in order to establish the {\it entropy-entropy dissipation inequality} (EEDI) \eqref{EEDI}. In all the previous works, the EEDI follows from the standard zero-average Poincar\'{e} inequality applied to the square roots of the concentration functions, combined with their uniform $L^\infty$-bounds (in space-time); these bounds need to be proved a priori.  This uniform $L^\infty$-bound is key to the proofs of  convergence to equilibrium in most of the works on this topic, (in fact, to our knowledge, the only exception comes when there are no boundary equilibria \cite{DFT2}) and the constant $\alpha$ from the EEDI \eqref{EEDI} tends to vanish as the $L^\infty$-bound on the solution blows up. We note that  \eqref{nonlinear-eq1} is one of the two systems studied in \cite{DF06}, and the authors use the uniform $L^\infty$-bound as available in the literature (for this particular system). In \cite{FatmaCA} the authors carry out the proof in some detail (adapted from a proof in \cite{ChenLiWright}), and show that the properties of the Neumann Heat Kernel involved in it hold for the discrete Neumann Heat Kernel as well; as a consequence, one can emulate the proof in the continuous case to obtain uniform $L^\infty$ bounds for the discretized problem. The main idea of the proof is a bootstrapping argument in which the bounds obtained on $a$ and $b$ in terms of $c$ (from the first two equations of \eqref{nonlinear-eq1} we get $a_t-k_a\Delta a\leq k_2c$ and $b_t-k_b\Delta b\leq k_2c$) are fed into the inequality $c_t-k_c\Delta c\leq k_1ab$ (from the third equation of \eqref{nonlinear-eq1}) to yield an $L^\infty$ bound on $c$ at some time $t\geq\delta>0$ in terms of a sublinear function of the bound at times $t\leq\delta/2$. The success of this  method relies on the right hand sides of the first two equations of \eqref{nonlinear-eq1} being bounded above in the positive orthant by a constant multiple of $c$. It therefore fails for systems with all multiple-species complexes (no single species in any complex) or systems where single species appear more than once in any given complex (such as $A+B{\rightleftharpoons}2C$). These bounds are crucial to the proof of consistency and, ultimately, convergence \cite{FatmaCA}.

Our method to prove these uniform $L^\infty$ bounds seems confined to one-D, as it uses the stronger form of Poincar\'{e}'s inequality on a bounded interval (where the essential sup norm of a Sobolev function is bounded in terms of its average and the $L^1$ norm of its Sobolev derivative). This leads to a uniform estimate (with respect to $t$) of the $L^2$ norm of the solutions in cylinders of type $(t,t+1) \times (0,1)$, which, once more using $d=1$, leads to a uniform in time $L^\infty$ bound in the case where one of the right hand side polynomials is bounded above by a quadratic polynomial. It is an important improvement that we can deal with the at most quadratic case, since previous results (when boundary equilibria are present) only dealt with two species reactions or, if at least three species are present, the right hand side of an equation from the system is dominated by a first-degree polynomial \cite{DF06}, \cite{DF08}, \cite{DF14}, \cite{DFPV07}, \cite{Fellner.2015aa} etc.

\subsection{The three-species system.}
A case not covered so far in the literature is $A+nB{\rightleftharpoons}B+C$ ($n\geq2$ is an integer); this has boundary equilibria in some (not all) stoichiometric classes, translates to a $3\times 3$ system ($2\times 2$ being, in general, easier to treat via the standard maximum principle for the heat equation), and the right hand side is not bounded above (in the positive orthant) by a linear term. More precisely, the PDE system we are looking at is
\begin{equation}\nonumber
\begin{cases}
a_t-d_a \Delta a=-k_1ab^n+k_2bc & \\ 
b_t-d_b \Delta b=-k_1ab^n+k_2bc &  \mbox{ in }\Omega\times(0,\infty) \\
c_t-d_c \Delta c=k_1ab^n-k_2bc  &   \\
\nabla{a} \cdot \nu=\nabla{b} \cdot \nu=\nabla{c} \cdot \nu=0 &  \mbox{ on }\partial\Omega\times(0,\infty)\\
a(\cdot,0)=a_0,\ b(\cdot,0)=b_0,\ c(\cdot,0)=c_0 &  \mbox{ in } \Omega, 
\end{cases}
\end{equation}
where $\nu$ is the (outward) normal vector to $\partial\Omega$.
Notice that by rescaling time $t$, space $x$ and the concentrations $(a,b,c)$, we can always assume that the reaction rates and the domain volume are 1; to fix ideas, let us choose, without loss of generality, $n=2$ and $k_1=k_2=1$. The only important restriction we impose is the choice of spatial dimension $d=1$. This is necessary due to the uniform $L^\infty$ bound on the solution, which is crucial to our analysis and we can (for the time being) only justify in the $d=1$ case; that is, only the estimates in subsection \ref{L2local} are predicated on this restriction. Therefore, the $3 \times 3$ reaction-diffusion system considered here is
\begin{equation}\label{3x3-system}
\begin{cases}
a_t-d_a a_{xx}=-ab^2+bc & \\ 
b_t-d_b b_{xx}=-ab^2+bc &  \mbox{ in }(0,1)\times(0,\infty) \\
c_t-d_c c_{xx}=ab^2-bc  &   \\
a_x=b_x=c_x=0 &  \mbox{ on }\{0,1\}\times(0,\infty)\\
a(\cdot,0)=a_0,\ b(\cdot,0)=b_0,\ c(\cdot,0)=c_0 &  \mbox{ in } (0,1). 
\end{cases}
\end{equation}
The conserved (in time) quantities here are $\bar a+\bar c$ and $\bar b+\bar c$, where $\bar f$ denotes the average of the function $f$ over $\Omega$. If $b_\infty>0$ we obviously can only have a boundary equilibrium at $(0,b_\infty,0)$ (i.e. $a_\infty=c_\infty=0$). The conservation of $\bar a+\bar c$ forces $a\equiv c\equiv 0$, the second equation of the system decouples into $b_t-d_b b_{xx}=0$, and $b_\infty=\bar b_0$. This steady state cannot be approached from any initial state for which $\bar a_0+\bar c_0>0$, so no initial data in the positive orthant will converge to it. The other nontrivial type of steady states is given by $b_\infty=0$ and $a_\infty+c_\infty=\bar a_0+\bar c_0>0$. If $c_\infty=0$, we get $b\equiv c\equiv 0$ (from the conservation of $\bar b+\bar c$), so, once again, an equilibrium of the type $(a_\infty,0,0)$ cannot be approached from the positive orthant.  We are left with the case $(a_\infty,0,c_\infty)$ for $a_\infty c_\infty>0$. We do not know how to prove that $\bar a_0\bar b_0 \bar c_0>0$ prevents convergence to such steady states, but in this paper we will prove a weaker statement, namely:
\begin{theorem}\label{conv-theorem1}
If $a_0,\ b_0, c_0\in L^\infty(0,1)$ are a.e. nonnegative and such that $\bar a_0\bar b_0 \bar c_0>0$ and $b_0\geq\delta$ a.e. in $(0,1)$ for some $\delta>0$, then the (unique) global classical solution to \eqref{3x3-system} converges asymptotically exponentially fast (at an explicit rate) to the unique positive equilibrium in its stoichiometric class.
\end{theorem}
The above theorem will be proved in Section \ref{3x3-Syst}.

\subsection{The two-species system}
Finally, in Section \ref{2x2-Syst} we prove similar results on the convergence to the positive equilibrium for a two-species reversible reaction-diffusion network with boundary equilibria: $$m_1A+n_1B{\rightleftharpoons}m_2A+n_2B.$$
Assume $m_1 > m_2 > 0$ and $0 < n_1 < n_2$ with $\bar{m} = m_1-m_2$, $\bar{n} = n_2-n_1$ and $\bar{m} < \bar{n}$.
The $2 \times 2$ reaction-diffusion system is
\begin{equation}\label{2x2-system}
\begin{cases}
a_t-d_a \Delta a= \bar m(a^{m_2}b^{n_2}-a^{m_1}b^{n_1}) &  \mbox{ in }\Omega\times(0,\infty) \\
b_t-d_b \Delta b= \bar n(a^{m_1}b^{n_1}-a^{m_2}b^{n_2}) &  \mbox{ in }\Omega\times(0,\infty) \\
\nabla{a} \cdot \nu=\nabla{b} \cdot \nu=0 &  \mbox{ on }\partial\Omega\times(0,\infty)\\
a(\cdot,0)=a_0,\ b(\cdot,0)=b_0 &   \mbox{ in } \Omega .
\end{cases}
\end{equation}

\begin{theorem}\label{conv-theorem2}
Let $\Omega$ be a bounded domain of $\mathbb{R}^d$ with a smooth boundary, for some integer $d\geq1$. If $ 0 < \alpha \leq a_0(x), 
b_0(x) \leq \beta < +\infty$ for a.e. $x$ in $\Omega$, then the (unique) global classical solution to \eqref{2x2-system} converges asymptotically exponentially (at an explicit rate) to the unique positive equilibrium in its stoichiometric class. 
\end{theorem}

\section{Asymptotic decay for the three-species system}\label{3x3-Syst}
We consider the entropy functional $E(a,b,c)$ and the corresponding entropy dissipation (when computed along solutions)
$D(a,b,c)=-\frac{d}{dt}{E(a,b,c)}$ associated to the system:

\begin{equation}\label{entropy} 
\begin{split}
E(a,b,c)=\int_{0}^{1}{a(\ln a-1)}dx+\int_{0}^{1}{b(\ln b-1)}dx+\int_{0}^{1}{c(\ln c-1)}dx
\end{split} 
\end{equation}
and 
\begin{equation}\label{entropy-dissipation} 
\begin{split}
& D(a,b,c)= 4d_{a}\int_{0}^{1}{|\partial_{x}{\sqrt{a}}|^2}dx+4d_{b}\int_{0}^{1}{|\partial_{x}{\sqrt{b}}|^2}dx+4d_{c}\int_{0}^{1}{|\partial_{x}{\sqrt{c}}|^2}dx
\\& +\int_{0}^{1}{(ab^2-bc)\ln(\frac{ab^2}{bc})dx}.
\end{split} 
\end{equation}
We would also like to record (for later use) the following {\it conservation laws}
\begin{equation}\label{conservation-laws}
\begin{split}
& \int_0^1a(x,t)dx+\int_0^1c(x,t)dx=\int_0^1a_0(x)dx+\int_0^1c_0(x)dx =:M_1,
\\ & \int_0^1b(x,t)dx+\int_0^1c(x,t)dx=\int_0^1b_0(x)dx+\int_0^1c_0(x)dx =: M_2,  \end{split} \end{equation}
for all $t\geq0$. Note that these are simply obtained by adding equations 1 and 3 (respectively, 2 and 3) and integrating in space over $[0,1]$ by taking into account the boundary conditions. Note also that $M_1$ and $M_2$ are finite as long as $a_0, b_0, c_0 \in L^1(0,1)$.

\subsection{Local $L^2$ estimate (a priori estimate)}\label{L2local}

\begin{proposition}\label{Local-bd}
Let $(a,b,c)$ be a solution for \eqref{3x3-system} with initial condition $(a_0,b_0,c_0)$ such that $a_0>0,\ b_0>0,\ c_0>0$ a.e. in $[0,1]$ and $a_0\ln{a_0}, b_0\ln{b_0}, c_0\ln{c_0} \in L^1(0,1)$. Then there exists a real constant $C$ such that

$$\| {a} \|_{L^{2}([0,1] \times [\tau,\tau+1])},\ \| {b} \|_{L^{2}([0,1] \times [\tau,\tau+1])},\ \| {c} \|_{L^{2}([0,1] \times [\tau,\tau+1])} \leq C$$
for any $\tau>0$. 
\end{proposition}

\begin{proof}

We start with the obvious inequality (which holds for all $x\in[0,1],\ t>0$)

\begin{equation*} 
\begin{split}
\bigg| \sqrt{a(x,t)}-\int_{0}^{1}\sqrt{a(y,t)}dy \bigg| \leq \int_{0}^{1}\big|\partial_{y} \sqrt{a(y,t)}\big|dy,
\end{split} 
\end{equation*}
then use H\"older's inequality to get
\begin{equation}\label{L-infinity-estimate} 
\begin{split}
a(x,t) & \leq \bigg(\int_{0}^{1}\sqrt{a(y,t)}dy+\int_{0}^{1}\big|\partial_{y} \sqrt{a(y,t)}\big|dy\bigg)^2 
\\& \leq 2\bigg(\int_{0}^{1}\sqrt{a(y,t)}dy\bigg)^2+2\bigg(\int_{0}^{1}\big|\partial_{y} \sqrt{a(y,t)}\big|dy\bigg)^2     
\\& \leq 2\int_{0}^{1}{a(y,t)}dy+2\int_{0}^{1}\big|\partial_{y} \sqrt{a(y,t)}\big|^2dy.
\end{split} 
\end{equation} 
Obviously, the above inequalities also hold for $b$ and $c$. Next we integrate the entropy dissipation in time to obtain 
\begin{equation*} 
 E(a(t),b(t),c(t))+  \int_{0}^{t}D(a(s),b(s),c(s))ds = E(a_0,b_0,c_0),
\end{equation*}
where we have only displayed the dependence on time of the components of the solution vector.
Since the last integrand in right hand side of \eqref{entropy-dissipation}  is nonnegative, we conclude

\begin{equation*} 
\begin{split}
& E(a(t),b(t),c(t))+  4d_{a}\int_{0}^{t}\int_{0}^{1}{|\partial_{x}{\sqrt{a}}|^2}dxdt+4d_{b}\int_{0}^{t}\int_{0}^{1}{|\partial_{x}{\sqrt{b}}|^2}dxdt+
\\ & 4d_{c}\int_{0}^{t}\int_{0}^{1}{|\partial_{x}{\sqrt{c}}|^2}dxdt 
\leq E(a_0,b_0,c_0).
\end{split} 
\end{equation*}

Since $x(\ln x-1) \geq -1$ for all $x\geq 0$ (at $x=0$ this holds in the limiting sense), we get

\begin{equation*} 
\begin{split}
& 4d_{a}\int_{0}^{t}\int_{0}^{1}{|\partial_{x}{\sqrt{a}}|^2}dxdt+4d_{b}\int_{0}^{t}\int_{0}^{1}{|\partial_{x}{\sqrt{b}}|^2}dxdt+4d_{c}\int_{0}^{t}\int_{0}^{1}{|\partial_{x}{\sqrt{c}}|^2}dxdt 
\\& \leq E(a_0,b_0,c_0)+3,   
\end{split} 
\end{equation*}
which implies
\begin{equation*} 
\begin{split}
& \| \partial_{x}{\sqrt{a}} \|_{L^{2}([0,1] \times [0,t])}^{2}+\| \partial_{x}{\sqrt{b}} \|_{L^{2}([0,1] \times [0,t])}^{2}+\| \partial_{x}{\sqrt{c}} \|_{L^{2}([0,1] \times [0,t])}^{2} 
\\& \leq \frac{E(a_0,b_0,c_0)+3}{4d} := \nolinebreak[4] C_1 
\end{split} 
\end{equation*} 
for any $t\geq0$, where $d := \min\{d_a,d_b,d_c\}$.
Finally, we take into account \eqref{conservation-laws} and \eqref{L-infinity-estimate} to estimate
\begin{equation*} 
\begin{split}
& \int_{\tau}^{\tau+1}\int_{0}^{1} {a^2}dxdt \leq 
\int_{\tau}^{\tau+1} {\| a(t) \|_{L^{\infty}[0,1]}} \bigg(\int_{0}^{1} {a(x,t)} dx\bigg)dt  
\\& \leq 2M_{1}\int_{\tau}^{\tau+1}\bigg(\int_{0}^{1}{a(x,t)}dx+\int_{0}^{1}|\partial_{x} \sqrt{a(x,t)}|^2dx\bigg)dt
\\&
\leq 2M_{1}\int_{\tau}^{\tau+1}\bigg(M_1+\int_{0}^{1}|\partial_{x} \sqrt{a(x,t)}|^2dx\bigg)dt 
\leq 2{M_1}^2+2M_1C_1=: C
\end{split} 
\end{equation*}

Similar inequalities hold for $b$ and $c$, therefore we have finished the proof. \end{proof}

\subsection{Uniform $L^{\infty}$ estimate}

In this section we shall prove that a classical solution to \eqref{3x3-system} is bounded uniformly in time (and therefore, it also exists for all time). To achieve this, our goal is to place ourselves in the setting of Theorem 4.1 \cite{FMS92}. We shall refrain from transcribing the assumptions (H1)--(H3) from \cite{FMS92} here, as they are universally satisfied by CRDN systems with nonnegative and essentially bounded initial conditions. On the other hand, assumption  (H4') is both specific to our case and nontrivial to verify. We state it below, as it refers to a general semilinear parabolic $m\times m$ system
\begin{equation}\label{gen-syst}
u_{i,t}-\kappa_i\Delta u_i=f_i(x,t,u),\ i=1,...,m,
\end{equation}
where $x\in\Omega$, $t>0$, $u=(u_1,...,u_m)$. It reads: 
$$\mbox{ There exist }K_1,\ K_2>0,\ 1\leq p<\infty,\ 1\leq r<1+\bigg[1-\frac{d}{p(d+2)}\bigg]\frac{2p}{d+2}\mbox{ such that }$$
\begin{equation}\label{H4} 
\mbox{ for each }1\leq j\leq m\mbox{ there exist }\alpha_{j,k},\ 1\leq j\leq k\mbox{ with }\alpha_{j,j}=1\mbox{ such that }
\end{equation}
$$\sum_{k=1}^j\alpha_{j,k}f_j(x,t,v)\leq K_1|v|^r+K_2\mbox{ for all }v\mbox{ in the positive orthant of }\mathbb{R}^m.$$
The following result is a version of Theorem 4.1 \cite{FMS92}. 
\begin{theorem}\label{thm-FMS92}
Suppose the initial data $u_{j,0}\in L^\infty(\Omega)$, $j=1,...,m$, the generic assumptions (H1)--(H3) from \cite{FMS92} hold. Further assume \eqref{H4} holds for some $1\leq p<\infty$ and 
\begin{equation}\label{cylinder-bound}
\|u\|_{L^p(\Omega\times(\tau,\tau+1);\mathbb{R}^m)}\leq M<\infty\mbox{ for all }\tau\geq0.
\end{equation}
Then 
\begin{equation}\label{unif-bound}
u\in L^\infty(\Omega\times[0,\infty);\mathbb{R}^m).
\end{equation}
\end{theorem}
A word of caution is in order: in \cite{FMS92} the analysis is performed on the whole space $\Omega=\mathbb{R}^d$ and for more general elliptic operators (instead of the Laplacian). However, we  argue that Theorem \ref{thm-FMS92} holds for bounded domains $\Omega$ as well. Indeed, our elliptic operator is the Laplacian and on a bounded domain $\Omega\subset\mathbb{R}^d$ we have that the estimate (used to prove Lemma 2.1 in \cite{FMS92}) on the fundamental solution to the corresponding parabolic operator 
$$0<G(x,\xi,t,\tau)\leq c_1(T)(t-\tau)^{-d/2}\exp\bigg\{-\frac{|x-\xi|^2}{c_2(T)(t-\tau)}\bigg\}$$
 for $x,\ \xi\in\Omega,\ 0<\tau<t\leq T<\infty$
(where $c_1(T),\ c_2(T)$ are bounded for finite $T$) holds for the Heat Kernel with Neumann BC on a bounded domain as well (see, e.g, \cite{ChoulliKayser}) with $c_1(t)=ce^t$, and $c_2(t)=C$, for some constants $c,\ C>0$. Likewise, for the homogeneous problem with initial value at $T\geq0$, we also have the estimate 
$$0<G_T(x,\xi,t)\leq c_1(t-T)(t-T)^{-d/2}\exp\bigg\{-\frac{|x-\xi|^2}{c_2(t-T)(t-T)}\bigg\},$$
which is used to prove Lemma 4.1 \cite{FMS92}. In the proof of Lemma 2.2 \cite{FMS92} there is an estimate on solutions of the adjoint equation; \cite{LadSolUral} is given as a reference. We limit ourselves to noting that \cite{LadSolUral} covers the case of and provides the same estimate on bounded domains as well. These are the estimates one needs to check in order to convince oneself that Theorem \ref{thm-FMS92} holds on a bounded domain $\Omega\subset{R}^d$. 

If we go back to \eqref{3x3-system} and denote by 
$$u:=(a,b,c),\ f:=(-ab^2+bc,-ab^2+bc,ab^2-bc)$$
we see that \eqref{H4} is satisfied with $\alpha_{1,1}=1, \alpha_{2,1}=-1,\ \alpha_{2,2}=1,\ \alpha_{3,1}=0,\ \alpha_{3,2}=-1,\ \alpha_{3,3}=1$, $r=2$, $K_1=k_2/2$, $K_2=0$. Thus, if we can find $1\leq p<\infty$ such that \eqref{cylinder-bound} and the first inequality in \eqref{H4} (the one bounding $r$ in terms of $p$) are satisfied, we can apply Theorem \ref{thm-FMS92} in order to obtain the uniform $L^\infty$ bound. But Proposition \ref{Local-bd} shows that $p=2$ does the job. 

The same reference \cite{FMS92} guarantees that the solution is classical, unique and nonnegative.
Therefore, we have proved:
\begin{theorem}\label{Linfty-thm}
If $a_0,\ b_0,\ c_0\in L^\infty(\Omega)$ and are a.e. nonnegative, then a unique classical solution to the system \eqref{3x3-system} exists for all time. Furthermore, the solution is uniformly (with respect to time) bounded in $L^\infty(\Omega)$.
\end{theorem}
\subsection{$L^1$ convergence}

Let us make the assumption $\beta = \|\frac{1}{b_0}\|_{L^{\infty}[0,1]} < \infty$; because the classical solution is continuous, there exists $t_1 > 0$ such that $\|\frac{1}{b(\cdot,t)}\|_{L^{\infty}[0,1]} < 10\beta$ for all $t \in [0,t_1]$. We next divide the second equation in \eqref{3x3-system} by $-b^2$ and use the uniform (in time) $L^{\infty}$ boundedness of $a$ to get
$$
\partial_t \bigg(\frac{1}{b}\bigg)-d_b \Delta \bigg(\frac{1}{b}\bigg)= \frac{ab^2}{b^2}-\frac{bc}{b^2}-2d_b\frac{|\nabla{b}|^2}{b^3} \leq a \leq k.
$$

Using the maximum principle for the heat equation, we have that, for all $t \in [0,t_1]$,
$$
\bigg\|\frac{1}{b(\cdot,t)}\bigg\|_{L^{\infty}[0,1]} \leq \bigg\|\frac{1}{b_0}\bigg\|_{L^{\infty}[0,1]} + kt = \beta +kt.
$$

We can iterate this inequality in time to get
\begin{equation}\label{tilde-b}
\tilde b(t):=\inf_{x \in [0,1]}{b(x,t)} \geq (\beta +kt)^{-1}
\end{equation}
for all $t>0$.
Therefore, we now have an estimate on how fast $b$ can decay to zero.

There exists a unique equilibrium with all positive components for \eqref{3x3-system} and by \eqref{3x3-system} and \eqref{conservation-laws} we see that it is given by $v_\infty:=(a_\infty,b_\infty,c_\infty)$, where its components are uniquely determined by
\begin{equation}\label{positive-equilibrium}
a_\infty b_\infty=c_\infty,\ a_\infty+c_\infty=M_1,\ b_\infty+c_\infty=M_2.
\end{equation}
Now we introduce the {\it relative entropy} 
\begin{equation}\label{relative-entropy} 
\begin{split}
& E(a,b,c|a_{\infty},b_{\infty},c_{\infty}) =  \int\limits_{[0,1]}\bigg({a\ln{\frac{a}{a_{\infty}}}}-a+a_{\infty}\bigg)dx+\int\limits_{[0,1]}\bigg({b\ln{\frac{b}{b_{\infty}}}}-b+b_{\infty}\bigg)dx
\\& +\int\limits_{[0,1]}\bigg({c\ln{\frac{c}{c_{\infty}}}}-c+c_{\infty}\bigg)dx
\end{split} 
\end{equation}
and its corresponding {\it entropy dissipation}
\begin{equation}\label{rel-entr-dissip} 
\begin{split}
& D(a,b,c|a_{\infty},b_{\infty},c_{\infty}) = d_a\int\limits_{[0,1]}{\frac{|\nabla{a}|^2}{a}}dx+d_b\int\limits_{[0,1]}{\frac{|\nabla{b}|^2}{b}}dx+d_c\int\limits_{[0,1]}{\frac{|\nabla{c}|^2}{c}}dx
\\& + a_{\infty}b_{\infty}^2\int\limits_{[0,1]}{\Psi\bigg(\frac{ab^2}{a_{\infty}b_{\infty}^2} ; \frac{bc}{b_{\infty}c_{\infty}}\bigg)}dx + b_{\infty}c_{\infty}\int\limits_{[0,1]}{\Psi\bigg( \frac{bc}{b_{\infty}c_{\infty}} ; \frac{ab^2}{a_{\infty}b_{\infty}^2}\bigg)}dx,
\end{split} 
\end{equation}
 where 
 \begin{equation}\label{Psi}
 \Psi(x;y)=x\ln\bigg(\frac{x}{y}\bigg)-x+y.
 \end{equation}
At this point we introduce the notation 
$$\bar f:=\int\limits_{[0,1]}f(x)dx\mbox{ for all essentially non-negative }f\in L^1(0,1).$$
On the basis of the following identity
$$
\int\limits_{[0,1]}\bigg({a\ln{\frac{a}{a_{\infty}}}}-a+a_{\infty}\bigg)dx = \int\limits_{[0,1]}\bigg({a\ln{\frac{a}{\overline{a}}}-a+\overline{a}}\bigg)dx+\int\limits_{[0,1]}\bigg({\overline{a}\ln{\frac{\overline{a}}{a_{\infty}}}}-\overline{a}+a_{\infty}\bigg)dx,
$$ 
we get 
\begin{equation}\label{rel-entropy-split}
E(a,b,c|a_{\infty},b_{\infty},c_{\infty}) = E(a,b,c|\overline{a},\overline{b},\overline{c}) + E(\overline{a},\overline{b},\overline{c}|a_{\infty},b_{\infty},c_{\infty}).
\end{equation}
The Logarithmic Sobolev Inequality
\begin{equation}\label{LSI}
\int\limits_{[0,1]}{\frac{|\nabla{f}|^2}{f}}dx \geq 
C_{LSI} \int\limits_{[0,1]}{f \ln{\frac{f}{\overline{f}}}}dx,
\end{equation}
(where $C_{LSI}$ only depends on the domain $[0,1]$) yields
\begin{equation}\label{LSI-1}
d_a\int\limits_{[0,1]}{\frac{|\nabla{a}|^2}{a}}dx+d_b\int\limits_{[0,1]}{\frac{|\nabla{b}|^2}{b}}dx+d_c\int\limits_{[0,1]}{\frac{|\nabla{c}|^2}{c}}dx \geq 
C_2  E(a,b,c|\overline{a},\overline{b},\overline{c})
\end{equation}
for an explicit constant $C_2=\min\{d_a,d_b,d_c\} \cdot C_{LSI}$. Next, we define two integrand functions:
$$
S_1(a,b,c) := \bigg({a\ln{\frac{a}{a_{\infty}}}}-a+a_{\infty}\bigg)+\bigg({b\ln{\frac{b}{b_{\infty}}}}-b+b_{\infty}\bigg)+\bigg({c\ln{\frac{c}{c_{\infty}}}}-c+c_{\infty}\bigg),
$$ 
$$
S_2(a,b,c):=\Psi\bigg(\frac{ab}{a_{\infty}b_{\infty}} ; \frac{c}{c_{\infty}}\bigg)+\Psi\bigg( \frac{c}{c_{\infty}} ; \frac{ab}{a_{\infty}b_{\infty}}\bigg)
$$ 
and set 
$S(a,b,c):=S_1(a,b,c)+S_2(a,b,c).$
From \eqref{rel-entr-dissip}, \eqref{rel-entropy-split}, \eqref{LSI-1} and \eqref{tilde-b} we get
\begin{equation}\label{big-ineq} 
\begin{split}
& D(a,b,c|a_{\infty},b_{\infty},c_{\infty}) \geq 
C_2 E(a,b,c|\overline{a},\overline{b},\overline{c})
\\& + a_{\infty}b_{\infty}^2\int\limits_{[0,1]}{\Psi{\bigg(\frac{ab^2}{a_{\infty}b_{\infty}^2} ; \frac{bc}{b_{\infty}c_{\infty}}\bigg)}}dx + b_{\infty}c_{\infty}\int\limits_{[0,1]}{\Psi{\bigg( \frac{bc}{b_{\infty}c_{\infty}} ; \frac{ab^2}{a_{\infty}b_{\infty}^2}\bigg)}}dx
\\& = C_2 \big[E(a,b,c|a_{\infty},b_{\infty},c_{\infty})-E(\overline{a},\overline{b},\overline{c}|a_{\infty},b_{\infty},c_{\infty})\big]
\\& +a_{\infty}b_{\infty}\tilde b(t)\int\limits_{[0,1]}{\Psi{\bigg(\frac{ab}{a_{\infty}b_{\infty}} ; \frac{c}{c_{\infty}}\bigg)}}dx 
+ c_{\infty}\tilde b(t)\int\limits_{[0,1]}{\Psi{\bigg( \frac{c}{c_{\infty}} ; \frac{ab}{a_{\infty}b_{\infty}}\bigg)}}dx
\\& \geq C_2 \int\limits_{[0,1]}{S(a,b,c)}dx-C_2 \int\limits_{[0,1]}{S(\overline{a},\overline{b},\overline{c})}dx
\\& +(\beta +kt)^{-1}\bigg[a_{\infty}b_{\infty}^2\int\limits_{[0,1]}{\Psi{\bigg(\frac{ab}{a_{\infty}b_{\infty}} ; \frac{c}{c_{\infty}}\bigg)}}dx 
+ b_{\infty}c_{\infty}\int\limits_{[0,1]}{\Psi{\bigg( \frac{c}{c_{\infty}} ; \frac{ab}{a_{\infty}b_{\infty}}\bigg)}}dx\bigg]
\\& \geq C_3(t) \bigg\{\int\limits_{[0,1]}{\big[S_1(a,b,c)+S_2(a,b,c)\big]}dx-S_1(\overline{a},\overline{b},\overline{c})\bigg\}
\\& \geq C_3(t)\big[{\hat{S}(\overline{a},\overline{b},\overline{c})}-S_1(\overline{a},\overline{b},\overline{c})\big],
\end{split} 
\end{equation}
where 
$$C_3(t):=(\beta +kt)^{-1} \min\{\beta C_2,a_{\infty}b_{\infty},c_{\infty}\}$$ and $\hat{S}$ is the {\it convexification} of $S$, i.e. the supremum of all affine functions below $S$. The last inequality above holds due to Jensen's inequality and the unit volume of the spatial domain.

We next define the compatible class:
\begin{equation*}
\begin{split}
   & C_{M_1,M_2}:= \big\{v=(x,y,z) \in {\mathbb{R}}_{\geq 0}^3: x+z=M_1,\  y+z=M_2,\\&
 E(x,y,z|a_{\infty},b_{\infty},c_{\infty}) 
    \leq E(a_0,b_0,c_0|a_{\infty},b_{\infty},c_{\infty})\big\}.
\end{split}
\end{equation*}
In this class, the first two conditions are related to the conservation laws \eqref{conservation-laws} while the last one follows from the decreasing relative entropy. Since we know $\hat{S}=\widehat{S_1+S_2} \geq \hat{S_1}+\hat{S_2}$, $S_1$ is convex and $S_2$ is non-negative, we have
\begin{equation}\label{S2-hat}
({\hat{S}}-S_1)(v) \geq (\hat{S_1}+\hat{S_2}-S_1)(v) \geq \hat{S_2}(v) \geq 0.
\end{equation}
Furthermore, it is not hard to verify that  
\begin{equation}\label{S2-zero}
v \in C_{M_1,M_2} \And S_2(v)=0 \mbox{ if and only if } 
v = (a_{\infty},b_{\infty},c_{\infty}).
\end{equation}
It follows 
$$\hat{S_2}(v) = 0\mbox{ if and only if }v = (a_{\infty},b_{\infty},c_{\infty}).$$ 
Let 
$$
C_4 := \inf_{v \in C_{M_1,M_2}}\frac{({\hat{S}}-S_1)(v)}{E(v|a_{\infty},b_{\infty},c_{\infty})}.
$$
By \eqref{S2-hat} and \eqref{S2-zero} we get $C_4$ can only be zero if there exists a sequence $\{v_n\}_n \subset C_{M_1,M_2}$ such that $v_n \rightarrow (a_{\infty},b_{\infty},c_{\infty})$ as $n\rightarrow\infty$. 
This means
$$
\liminf_{v \in C_{M_1,M_2},v \rightarrow v_\infty}{\frac{({\hat{S}}-S_1)(v)}{E(v|a_{\infty},b_{\infty},c_{\infty})}} > 0 \ \mbox{ implies } C_4 > 0.
$$

In order to show that the above limit inferior is positive we use the following lemma \cite{MHM}:
\begin{lemma}
  There exists $\delta > 0$ such that for all $ v \in B(v_\infty,\delta)$ (ball centered at $v_\infty$ and of radius $delta$) $S(v)$ is locally convex in this ball. 
 \end{lemma}
In particular, we get that $\hat{S} \equiv S$ in the ball centered at $(a_{\infty},b_{\infty},c_{\infty})$ with radius $\delta$.
Let us now define 
\begin{equation}\nonumber
 D_2(v):= a_{\infty}b_{\infty}^2 \Psi\bigg(\frac{ab^2}{a_{\infty}b_{\infty}^2};\frac{bc}{b_{\infty}c_{\infty}}\bigg) + b_{\infty}c_{\infty}\Psi\bigg( \frac{bc}{b_{\infty}c_{\infty}} ; \frac{ab^2}{a_{\infty}b_{\infty}^2}\bigg) 
\end{equation}
and consider the Taylor expansion of 
$$
\frac{D_2(v)}{E(v|a_{\infty},b_{\infty},c_{\infty})}
$$
around the unique positive equilibrium $(a_{\infty},b_{\infty},c_{\infty})$. Since $a_{\infty}b_{\infty}=c_{\infty}$, we have $D_2(a_{\infty},b_{\infty},c_{\infty})=\nabla{D_2(a_{\infty},b_{\infty},c_{\infty})}=0$ and quadratic term in the expansion is 
$$
D_2(v)= 2\bigg[\frac{-(v_1-a_{\infty})}{a_{\infty}}+\frac{-(v_2-b_{\infty})}{b_{\infty}}+\frac{(v_3-c_{\infty})}{c_{\infty}}\bigg]^2.
$$
Thus, 
$$
\liminf_{v \in C_{M_1,M_2},v \rightarrow v_\infty}\frac{D_2(v)}{E(v|a_{\infty},b_{\infty},c_{\infty})} = \inf_{v \in C_{M_1,M_2}}{\frac{2\big[\frac{-(x-a_{\infty})}{a_{\infty}}+\frac{-(y-b_{\infty})}{b_{\infty}}+\frac{(z-c_{\infty})}{c_{\infty}}\big]^2}{\frac{(x-a_{\infty})^2}{a_{\infty}}+\frac{(y-b_{\infty})^2}{b_{\infty}}+\frac{(z-c_{\infty})^2}{c_{\infty}}}}
$$
Since $v \in C_{M_1,M_2}$ (which means $ x+z=a_{\infty}+c_{\infty}, y+z=b_{\infty}+c_{\infty}$), we get
\begin{equation*}
-(x-a_{\infty})=-(y-b_{\infty})=z-c_{\infty},
\end{equation*}
Then 
\begin{equation*}
\inf_{v \in C_{M_1,M_2}}{\frac{2\big[\frac{-(x-a_{\infty})}{a_{\infty}}+\frac{-(y-b_{\infty})}{b_{\infty}}+\frac{(z-c_{\infty})}{c_{\infty}}\big]^2}{\frac{(x-a_{\infty})^2}{a_{\infty}}+\frac{(y-b_{\infty})^2}{b_{\infty}}+\frac{(z-c_{\infty})^2}{c_{\infty}}}} = 2\bigg(\frac{1}{a_{\infty}}+\frac{1}{b_{\infty}}+\frac{1}{c_{\infty}}\bigg) > 0.
\end{equation*}
Also notice (by direct computation and using that $c_\infty=a_\infty b_\infty$) the identity $D_2(v) = bc_{\infty}S_2(v)$, which implies (in view of the above inequality)
$$
\liminf_{v \in C_{M_1,M_2},v \rightarrow v_\infty}{\frac{S_2(v)}{E(v|a_{\infty},b_{\infty},c_{\infty})}} > 0.
$$

Combining the above two steps, we have 
\begin{equation*}
    \begin{split}
         & \liminf_{v \in C_{M_1,M_2},v \rightarrow v_\infty}{\frac{({\hat{S}}-S_1)(v)}{E(v|a_{\infty},b_{\infty},c_{\infty})}} 
        \\& = \liminf_{v \in C_{M_1,M_2},v \rightarrow v_\infty}{\frac{S_2(v)}{E(v|a_{\infty},b_{\infty},c_{\infty})}} > 0.
    \end{split}
\end{equation*}

Therefore, in light of \eqref{big-ineq}, we obtain
$$
D(a,b,c|a_{\infty},b_{\infty},c_{\infty}) \geq C_3(t)C_4E(a,b,c|a_{\infty},b_{\infty},c_{\infty})
$$
so, 
$$
D(a,b,c|a_{\infty},b_{\infty},c_{\infty}) \geq C_5(\beta +kt)^{-1}E(a,b,c|a_{\infty},b_{\infty},c_{\infty}),
$$
where $C_5 = \min\{1,C_4\} \times \min\{\beta C_2,a_{\infty}b_{\infty}^2,b_{\infty}c_{\infty}\}$.
Then Gronwall's lemma yields
$$
  E(a,b,c|a_{\infty},b_{\infty},c_{\infty}) \leq E(a_0,b_0,c_0|a_{\infty},b_{\infty},c_{\infty})(\beta +kt)^{\frac{-C_5}{k}}
$$
for all $t>0$.

Now we need the following lemma \cite{AMTU}:

\begin{lemma}
For all non-negative and measurable functions $a,b,c: [0,1] \rightarrow \mathbb{R}$ and $\int_{0}^{1}(a+c)=M_1, \int_{0}^{1}(b+c)=M_2$. Then there exists a constant $C_K >0$ related only with domain and $M_1, M_2$ such that:
$$
E(a,b,c|a_{\infty},b_{\infty},c_{\infty}) \geq C_K (\|a-a_{\infty}\|_{1}^2+\|b-b_{\infty}\|_{1}^2+\|c-c_{\infty}\|_{1}^2)
$$
\end{lemma}
Therefore, we get
$$
\|a-a_{\infty}\|_{1}^2+\|b-b_{\infty}\|_{1}^2+\|c-c_{\infty}\|_{1}^2 \leq
C_6(\beta +kt)^{\frac{-C_5}{k}},
$$
where $C_6 = \frac{E(a_0,b_0,c_0|a_{\infty},b_{\infty},c_{\infty})}{C_K}$.

The above inequality shows that the solution stays away from the boundary equilibrium; in fact, its converges to the unique positive equilibrium in the $L^1$ norm. In order to show that the convergence rate is, in fact, exponential, we use the above inequality to conclude that there exists a time $T_{\epsilon}$ such that $\|a(t)\|_{1}, \|b(t)\|_{1}, \|c(t)\|_{1}  > \epsilon^{2} > 0$ (for some sufficiently small $\epsilon < 1$) for all $t > T_{\epsilon}$.

\subsection{Entropy entropy-dissipation estimate}
By using the inequality \eqref{Psi-ineq}, we obtain
\begin{equation}\label{C7} 
\begin{split}
& D(a,b,c|a_{\infty},b_{\infty},c_{\infty}) = d_a\int\limits_{[0,1]}{\frac{|\nabla{a}|^2}{a}}dx+d_b\int\limits_{[0,1]}{\frac{|\nabla{b}|^2}{b}}dx+d_c\int\limits_{[0,1]}{\frac{|\nabla{c}|^2}{c}}dx
\\& + a_{\infty}b_{\infty}^2\int\limits_{[0,1]}{\Psi{\bigg(\frac{ab^2}{a_{\infty}b_{\infty}^2} ; \frac{bc}{b_{\infty}c_{\infty}}\bigg)}}dx + b_{\infty}c_{\infty}\int\limits_{[0,1]}{\Psi{\bigg( \frac{bc}{b_{\infty}c_{\infty}} ; \frac{ab^2}{a_{\infty}b_{\infty}^2}\bigg)}}dx
\\& \geq 4d_a\|\nabla{\sqrt{a}}\|_2^{2}+4d_b\|\nabla{\sqrt{b}}\|_2^{2}+4d_c\|\nabla{\sqrt{c}}\|_2^{2}
\\& + a_{\infty}b_{\infty}^2\bigg\|\sqrt{\frac{ab^2}{a_{\infty}b_{\infty}^2}}   - \sqrt{\frac{bc}{b_{\infty}c_{\infty}}}\bigg\|_2^{2} + b_{\infty}c_{\infty}\bigg\| \sqrt{\frac{bc}{b_{\infty}c_{\infty}}} - \sqrt{\frac{ab^2}{a_{\infty}b_{\infty}^2}}\bigg\|_2^{2}
\\& \geq C_7\bigg(\|\nabla{\sqrt{a}}\|_2^{2}+\|\nabla{\sqrt{b}}\|_2^{2}+\|\nabla{\sqrt{c}}\|_2^{2}+\bigg\|\sqrt{\frac{ab^2}{a_{\infty}b_{\infty}^2}}   - \sqrt{\frac{bc}{b_{\infty}c_{\infty}}}\bigg\|_2^{2}\bigg),
\end{split} 
\end{equation}
where $C_7 := \min(4d_a,4d_b,4d_c,a_{\infty}b_{\infty}^2+b_{\infty}c_{\infty})$. Due to \eqref{conservation-laws}, we have $M := \max(M_1,M_2)$ such that $\overline{a(t)},\overline{b(t)},\overline{c(t)} < M$ for all $t\geq0$. In what follows we drop the dependence on $t$ from the notation; each time we write $\overline{a}$ or the likes we mean the spatial average of $a(t)=a(\cdot,t)$.
Due to \eqref{psi-increase}, we see that
$$
 \Psi(x,y) \leq \frac{\Psi(M,y)}{(\sqrt{M}-\sqrt{y})^2}(\sqrt{x}-\sqrt{y})^2\mbox{ for all }x\leq M.
$$

Since $0 < a_{\infty},b_{\infty},c_{\infty} < M$, we have
\begin{equation}\label{C8} 
\begin{split}
& E(\overline{a},\overline{b},\overline{c}|a_{\infty},b_{\infty},c_{\infty}) = \bigg({\overline{a}\ln{\frac{\overline{a}}{a_{\infty}}}}-\overline{a}+a_{\infty}\bigg) + \bigg({\overline{b}\ln{\frac{\overline{b}}{b_{\infty}}}}-\overline{b}+b_{\infty}\bigg) 
\\& + \bigg({\overline{c}\ln{\frac{\overline{c}}{c_{\infty}}}}-\overline{c}+c_{\infty}\bigg) < \frac{\Psi(M,a_{\infty})}{(\sqrt{M}-\sqrt{a_{\infty}})^2}(\sqrt{\overline{a}}-\sqrt{a_{\infty}})^2 
\\& + \frac{\Psi(M,b_{\infty})}{(\sqrt{M}-\sqrt{b_{\infty}})^2}(\sqrt{\overline{b}}-\sqrt{b_{\infty}})^2 
+ \frac{\Psi(M,c_{\infty})}{(\sqrt{M}-\sqrt{c_{\infty}})^2}(\sqrt{\overline{c}}-\sqrt{c_{\infty}})^2
\\& \leq C_8\big[(\sqrt{\overline{a}}-\sqrt{a_{\infty}})^2 + (\sqrt{\overline{b}}-\sqrt{b_{\infty}})^2 + (\sqrt{\overline{c}}-\sqrt{c_{\infty}})^2 \big],
\end{split} 
\end{equation}
where $$C_8 := \max\bigg\{\frac{\Psi(M,a_{\infty})}{(\sqrt{M}-\sqrt{a_{\infty}})^2},\frac{\Psi(M,b_{\infty})}{(\sqrt{M}-\sqrt{b_{\infty}})^2},\frac{\Psi(M,c_{\infty})}{(\sqrt{M}-\sqrt{c_{\infty}})^2}\bigg\}.$$


Next we claim that there exists a real constant $C_9$ such that
\begin{equation}\label{C9}
\begin{split}
& \|\nabla{\sqrt{a}}\|_2^{2}+\|\nabla{\sqrt{b}}\|_2^{2}+\|\nabla{\sqrt{c}}\|_2^{2}+\bigg\|\sqrt{\frac{ab^2}{a_{\infty}b_{\infty}^2}}   - \sqrt{\frac{bc}{b_{\infty}c_{\infty}}}\bigg\|_2^{2} 
\\& > C_9\bigg[ \|\nabla{\sqrt{a}}\|_2^{2}+\|\nabla{\sqrt{b}}\|_2^{2}+\|\nabla{\sqrt{c}}\|_2^{2}+\bigg(\frac{\overline{\sqrt{a}}\overline{\sqrt{b}}^2}{\sqrt{a_{\infty}b_{\infty}^2}} - \frac{\overline{\sqrt{b}}\overline{\sqrt{c}}}{\sqrt{b_{\infty}c_{\infty}}}\bigg)^{2}\bigg].    
\end{split}
\end{equation}

In order to get the above estimate, we introduce the deviations from the mean, i.e. $\delta_a = \sqrt{a}-\overline{\sqrt{a}},\delta_b = \sqrt{b}-\overline{\sqrt{b}},\delta_c = \sqrt{c}-\overline{\sqrt{c}}$.
Now we make the decomposition 
$$
[0,1] = D_L \cup D_L^{\complement},
$$
where $D_L = \{ x \in [0,1]: |\delta_a|,|\delta_b|,|\delta_c| \leq L \}$ for a fixed constant $L$. We expand 
\begin{equation*}
\begin{split}
& \sqrt{ab^2} = \big(\overline{\sqrt{a}}+\delta_a\big)\big(\overline{\sqrt{b}}+\delta_b\big)^2 = \overline{\sqrt{a}}\overline{\sqrt{b}}^2 + \big[\delta_a\big(\overline{\sqrt{b}}+\delta_b\big)^2 + \overline{\sqrt{a}}\big(2\overline{\sqrt{b}}\delta_b+\delta_b^2\big)\big]
\end{split}
\end{equation*}
and
\begin{equation*}
\begin{split}
\sqrt{bc} &= \big(\overline{\sqrt{b}}+\delta_b\big)\big(\overline{\sqrt{c}}+\delta_c\big) = \overline{\sqrt{b}}\overline{\sqrt{c}} + \big[\delta_b\overline{\sqrt{c}} + \delta_c\big(\overline{\sqrt{b}}+\delta_b\big)\big]
\end{split}
\end{equation*}
to see that on the set $D_L$ one has
\begin{equation*}
\begin{split}
&\delta_a\big(\overline{\sqrt{b}}+\delta_b\big)^2 + \overline{\sqrt{a}}\big(2\overline{\sqrt{b}}\delta_b+\delta_b^2\big) 
\\& \leq  (|\delta_a|+|\delta_b|)\big[\big(\sqrt{M_2}+L\big)^2 + \sqrt{M_1}\big(2\sqrt{M_2}+L\big)\big]
 = (|\delta_a|+|\delta_b|)R_1
\end{split}
\end{equation*}
and
\begin{equation*}
\begin{split}
& \delta_b\overline{\sqrt{c}} + \delta_c\big(\overline{\sqrt{b}}+\delta_b\big)
\\& \leq  (|\delta_b|+|\delta_c|)\big[\sqrt{M_2} + (\sqrt{M_2}+L)\big]
= (|\delta_b|+|\delta_c|)R_2,
\end{split}
\end{equation*}
where $R_1 := \big(\sqrt{M_2}+L\big)^2 + \sqrt{M_1}\big(2\sqrt{M_2}+L\big)$ and $R_2 := \sqrt{M_2} + \big(\sqrt{M_2}+L\big)$.
Thus, 
\begin{equation}\label{R-constant}
\begin{split}
& \bigg\|\sqrt{\frac{ab^2}{a_{\infty}b_{\infty}^2}}   - \sqrt{\frac{bc}{b_{\infty}c_{\infty}}}\bigg\|_{L^{2}(D_L)}^{2}
 = \bigg\|\frac{\overline{\sqrt{a}}\overline{\sqrt{b}}^2}{\sqrt{a_{\infty}b_{\infty}^2}} - \frac{\overline{\sqrt{b}}\overline{\sqrt{c}}}{\sqrt{b_{\infty}c_{\infty}}} 
 \\&  + \frac{[\delta_a(\overline{\sqrt{b}}+\delta_b)^2 + \overline{\sqrt{a}}(2\overline{\sqrt{b}}\delta_b+\delta_b^2)]}{{\sqrt{a_{\infty}b_{\infty}^2}}} - \frac{ [\delta_b\overline{\sqrt{c}} + \delta_c(\overline{\sqrt{b}}+\delta_b)]}{\sqrt{b_{\infty}c_{\infty}}} \bigg\|_{L^{2}(D_L)}^{2}
\\& \geq \frac{1}{2}\bigg(\frac{\overline{\sqrt{a}}\overline{\sqrt{b}}^2}{\sqrt{a_{\infty}b_{\infty}^2}} - \frac{\overline{\sqrt{b}}\overline{\sqrt{c}}}{\sqrt{b_{\infty}c_{\infty}}}\bigg)^2|D_L| - 2\big\||\delta_a|+|\delta_b|\big\|_{L^{2}(D_L)}^{2}\frac{R_1^2}{a_{\infty}b_{\infty}^2} 
\\& - 2\big\||\delta_b|+|\delta_c|\big\|_{L^{2}(D_L)}^{2}\frac{R_2^2}{b_{\infty}c_{\infty}}
\\& \geq \frac{1}{2}\bigg(\frac{\overline{\sqrt{a}}\overline{\sqrt{b}}^2}{\sqrt{a_{\infty}b_{\infty}^2}} - \frac{\overline{\sqrt{b}}\overline{\sqrt{c}}}{\sqrt{b_{\infty}c_{\infty}}}\bigg)^2|D_L| - R(M_1,M_2,L)\big[\|\delta_a\|_{L^{2}(D_L)}^{2}+
\\& \|\delta_b\|_{L^{2}(D_L)}^{2}+\|\delta_c\|_{L^{2}(D_L)}^{2}\big],
\end{split} 
\end{equation}
\\
where $R(M_1,M_2,L) := \frac{4R_1^2}{a_{\infty}b_{\infty}^2}  + \frac{4R_2^2}{b_{\infty}c_{\infty}}$. 

On the set $D_L^{\complement}$, by using Poincar\'{e}'s inequality, we get 
\begin{equation*}
\begin{split}
& \|\nabla{\sqrt{a}}\|_2^{2}+\|\nabla{\sqrt{b}}\|_2^{2}+\|\nabla{\sqrt{c}}\|_2^{2} 
\\& \geq C_P (\|\delta_a\|_{L^{2}(D_L^{\complement})}^{2}+ \|\delta_b\|_{L^{2}(D_L^{\complement})}^{2}+\|\delta_c\|_{L^{2}(D_L^{\complement})}^{2})
\\& \geq C_P L^2 |D_L^{\complement}|. 
\end{split}
\end{equation*}
Since $$\bigg(\frac{\overline{\sqrt{a}}\overline{\sqrt{b}}^2}{\sqrt{a_{\infty}b_{\infty}^2}}  -  \frac{\overline{\sqrt{b}}\overline{\sqrt{c}}}{\sqrt{b_{\infty}c_{\infty}}}\bigg)^2 \leq \bigg(\frac{\sqrt{M_1}\sqrt{M_2}^2}{\sqrt{a_{\infty}b_{\infty}^2}}  +  \frac{\sqrt{M_2}\sqrt{M_2}}{\sqrt{b_{\infty}c_{\infty}}}\bigg)^2,$$ we infer 
\begin{equation}\label{tilde-R}
\|\nabla{\sqrt{a}}\|_2^{2}+\|\nabla{\sqrt{b}}\|_2^{2}+\|\nabla{\sqrt{c}}\|_2^{2} \geq \Tilde{R} \bigg(\frac{\overline{\sqrt{a}}\overline{\sqrt{b}}^2}{\sqrt{a_{\infty}b_{\infty}^2}}  -  \frac{\overline{\sqrt{b}}\overline{\sqrt{c}}}{\sqrt{b_{\infty}c_{\infty}}}\bigg)^2|D_L^{\complement}|,
\end{equation}
where $$\Tilde{R} := \frac{C_P L^2}{\bigg(\frac{\sqrt{M_1}\sqrt{M_2}^2}{\sqrt{a_{\infty}b_{\infty}^2}}  +  \frac{\sqrt{M_2}\sqrt{M_2}}{\sqrt{b_{\infty}c_{\infty}}}\bigg)^2}.$$
Pick $K > \frac{R+1}{\min\{1,C_P\}}$ and combine \eqref{R-constant} and \eqref{tilde-R} to conclude 
\begin{equation*}
\begin{split}
& 3K\big(\|\nabla{\sqrt{a}}\|_2^{2}+\|\nabla{\sqrt{b}}\|_2^{2}+\|\nabla{\sqrt{c}}\|_2^{2}\big) +  \bigg\|\sqrt{\frac{ab^2}{a_{\infty}b_{\infty}^2}}  - \sqrt{\frac{bc}{b_{\infty}c_{\infty}}}\bigg\|_2^{2} 
\\& \geq K\big(\|\nabla{\sqrt{a}}\|_2^{2}+\|\nabla{\sqrt{b}}\|_2^{2}+\|\nabla{\sqrt{c}}\|_2^{2}\big) + K\Tilde{R} \bigg(\frac{\overline{\sqrt{a}}\overline{\sqrt{b}}^2}{\sqrt{a_{\infty}b_{\infty}^2}}  -  \frac{\overline{\sqrt{b}}\overline{\sqrt{c}}}{\sqrt{b_{\infty}c_{\infty}}}\bigg)^2|D_L^{\complement}| \ + 
\\&  \bigg\{ \frac{1}{2}\bigg(\frac{\overline{\sqrt{a}}\overline{\sqrt{b}}^2}{\sqrt{a_{\infty}b_{\infty}^2}}  - \frac{\overline{\sqrt{b}}\overline{\sqrt{c}}}{\sqrt{b_{\infty}c_{\infty}}}\bigg)^2|D_L| -
 R\big(\|\delta_a\|_{L^{2}(D_L)}^{2}+ \|\delta_b\|_{L^{2}(D_L)}^{2}+\|\delta_c\|_{L^{2}(D_L)}^{2}\big) \bigg\}
\\&  + K\big(\|\delta_a\|_{L^{2}(D_L)}^{2}+ \|\delta_b\|_{L^{2}(D_L)}^{2}+\|\delta_c\|_{L^{2}(D_L)}^{2}\big)
\\& \geq K\big(\|\nabla{\sqrt{a}}\|_2^{2}+\|\nabla{\sqrt{b}}\|_2^{2}+\|\nabla{\sqrt{c}}\|_2^{2}\big) + \min\bigg\{K\Tilde{R},\frac{1}{2}\bigg\}\bigg(\frac{\overline{\sqrt{a}}\overline{\sqrt{b}}^2}{\sqrt{a_{\infty}b_{\infty}^2}}  -  \frac{\overline{\sqrt{b}}\overline{\sqrt{c}}}{\sqrt{b_{\infty}c_{\infty}}}\bigg)^2
\\& + (KC_P-R)\big(\|\delta_a\|_{L^{2}(D_L)}^{2}+ \|\delta_b\|_{L^{2}(D_L)}^{2}+\|\delta_c\|_{L^{2}(D_L)}^{2}\big)
\\& \geq C_{K,R}\bigg[ \|\nabla{\sqrt{a}}\|_2^{2}+\|\nabla{\sqrt{b}}\|_2^{2}+\|\nabla{\sqrt{c}}\|_2^{2}+\bigg(\frac{\overline{\sqrt{a}}\overline{\sqrt{b}}^2}{\sqrt{a_{\infty}b_{\infty}^2}} - \frac{\overline{\sqrt{b}}\overline{\sqrt{c}}}{\sqrt{b_{\infty}c_{\infty}}}\bigg)^{2}\bigg]. 
\end{split}
\end{equation*}
where $ C_{K,R} = \min\big\{K,K\Tilde{R},\frac{1}{2},KC_P-R\big\}=\min\big\{K\Tilde{R},\frac{1}{2}\big\}$ (because $K-R>1$).
Therefore, \eqref{C9} is proved with $C_9 = \frac{C_{K,R}}{3K}$.\\

It remains to show that there exists a constant $C_{10}$ such that
\begin{equation}\label{C10}
\begin{split}
 & \|\nabla{\sqrt{a}}\|_2^{2}+\|\nabla{\sqrt{b}}\|_2^{2}+\|\nabla{\sqrt{c}}\|_2^{2}+\bigg(\frac{\overline{\sqrt{a}}\overline{\sqrt{b}}^2}{\sqrt{a_{\infty}b_{\infty}^2}} - \frac{\overline{\sqrt{b}}\overline{\sqrt{c}}}{\sqrt{b_{\infty}c_{\infty}}}\bigg)^{2} \geq 
 \\& C_{10}\big[\big(\sqrt{\overline{a}}-\sqrt{a_{\infty}}\big)^2 + \big(\sqrt{\overline{b}}-\sqrt{b_{\infty}}\big)^2 + \big(\sqrt{\overline{c}}-\sqrt{c_{\infty}}\big)^2\big].
\end{split} 
\end{equation}
To this end, we introduce $\mu_a,\mu_b,\mu_c$ to parameterize
$\sqrt{\overline{a}},\sqrt{\overline{b}},\sqrt{\overline{c}}$ with 
$\sqrt{\overline{a}}=\sqrt{a_{\infty}}(1+\mu_a)$,$\sqrt{\overline{b}}=\sqrt{b_{\infty}}(1+\mu_b)$,$\sqrt{\overline{c}}=\sqrt{c_{\infty}}(1+\mu_c)$, where $-1 \leq \mu_a,\mu_b,\mu_c < \mu_k$ for $\mu_k = \frac{\sqrt{k}}{\min\{\sqrt{a_{\infty}},\sqrt{b_{\infty}},\sqrt{c_{\infty}}\}}-1$. Since $\delta_a = \sqrt{a} - \overline{\sqrt{a}}$, we have 
\begin{equation*}
    \begin{split}
& \|\delta_a\|_2^2 = \overline{a} - (\overline{\sqrt{a}})^2 = (\sqrt{\overline{a}} - \overline{\sqrt{a}})(\sqrt{\overline{a}} + \overline{\sqrt{a}})
\\& \implies \overline{\sqrt{a}} = -\frac{\|\delta_a\|_2^2}{\sqrt{\overline{a}} + \overline{\sqrt{a}}} + \sqrt{\overline{a}} = \sqrt{\overline{a}} - T(a)\|\delta_a\|_2^2,
    \end{split}
\end{equation*}
where $T(a) = \frac{1}{\sqrt{\overline{a}} + \overline{\sqrt{a}}} \leq \frac{1}{\epsilon}$; this inequality follows from $\overline{a}=\|a\|_{1} > \epsilon^{2} > 0$.
Similarly,
$$
\overline{\sqrt{b}}= \sqrt{\overline{b}} - T(b)\|\delta_b\|_2^2 \And   \overline{\sqrt{c}}= \sqrt{\overline{c}} - T(c)\|\delta_c\|_2^2,
$$
where $T(b) = \frac{1}{\sqrt{\overline{b}} + \overline{\sqrt{b}}},T(c) = \frac{1}{\sqrt{\overline{c}} + \overline{\sqrt{c}}} \leq \frac{1}{\epsilon}$.
And since  
$$
\epsilon^2 < \|b\|_{1} \leq \|\sqrt{b}\|_{\infty}\|\sqrt{b}\|_{1} \leq \sqrt{k}\|\sqrt{b}\|_{1} \implies \overline{\sqrt{b}} \geq \frac{\epsilon^2}{\sqrt{k}},
$$
due to this lower bound on $\overline{\sqrt{b}}$, we can factor out $(\frac{\overline{\sqrt{b}}}{\sqrt{b_{\infty}}})^2$ and reduce \eqref{3x3-system} to the system associated with the reversible reaction $A+B{\rightleftharpoons}C$ (which does not have boundary equilibria).
More precisely, we have
\begin{equation}\label{Bigg-Ineq}
    \begin{split}
& \bigg(\frac{\overline{\sqrt{a}}\overline{\sqrt{b}}^2}{\sqrt{a_{\infty}b_{\infty}^2}} - \frac{\overline{\sqrt{b}}\overline{\sqrt{c}}}{\sqrt{b_{\infty}c_{\infty}}}\bigg)^{2} = \frac{\big(\overline{\sqrt{b}}\big)^2}{b_{\infty}}\, \bigg(\frac{\overline{\sqrt{a}}\overline{\sqrt{b}}}{\sqrt{a_{\infty}b_{\infty}}} - \frac{\overline{\sqrt{c}}}{\sqrt{c_{\infty}}}\bigg)^{2} 
\\& \geq \frac{\epsilon^4}{b_{\infty}k}\bigg\{\frac{(\sqrt{\overline{a}} - T(a)\|\delta_a\|_2^2)(\sqrt{\overline{b}} - T(b)\|\delta_b\|_2^2)}{\sqrt{a_{\infty}b_{\infty}}} - \frac{\sqrt{\overline{c}} - T(c)\|\delta_c\|_2^2}{\sqrt{c_{\infty}}}\}^{2} 
\\& = \frac{\epsilon^4}{b_{\infty}k}
\{[1+\mu_a - \frac{T(a)\|\delta_a\|_2^2}{\sqrt{a_{\infty}}}][1+\mu_b - \frac{T(b)\|\delta_b\|_2^2}{\sqrt{b_{\infty}}}] - [1+\mu_c - \frac{ T(c)\|\delta_c\|_2^2}{\sqrt{c_{\infty}}}]\bigg\}^{2} 
\\& = \frac{\epsilon^4}{b_{\infty}k}
\bigg\{[(1+\mu_a)(1+\mu_b)-(1+\mu_c)]+\frac{T(a)\|\delta_a\|_2^2T(b)\|\delta_b\|_2^2}{\sqrt{a_{\infty}}} + \frac{ T(c)\|\delta_c\|_2^2}{\sqrt{c_{\infty}}}
\\& -\frac{T(a)\|\delta_a\|_2^2}{\sqrt{a_{\infty}}}(1+\mu_b)-\frac{T(b)\|\delta_b\|_2^2}{\sqrt{b_{\infty}}}(1+\mu_a)\bigg\}^2
\\& \geq \frac{\epsilon^4}{b_{\infty}k}
\bigg\{\frac{1}{2}[(1+\mu_a)(1+\mu_b)-(1+\mu_c)]^2-\bigg[\frac{T(a)\|\delta_a\|_2^2T(b)\|\delta_b\|_2^2}{\sqrt{a_{\infty}}} + \frac{ T(c)\|\delta_c\|_2^2}{\sqrt{c_{\infty}}}
\\& -\frac{T(a)\|\delta_a\|_2^2}{\sqrt{a_{\infty}}}(1+\mu_b)-\frac{T(b)\|\delta_b\|_2^2}{\sqrt{b_{\infty}}}(1+\mu_a)\bigg]^2\bigg\}
\\& \geq \frac{\epsilon^4}{b_{\infty}k}
\bigg\{\frac{1}{2}\big[(1+\mu_a)(1+\mu_b)-(1+\mu_c)\big]^2-4\bigg[\frac{T(a)\|\delta_a\|_2^2T(b)\|\delta_b\|_2^2}{\sqrt{a_{\infty}}}\bigg]^2\\
& - 4\bigg[\frac{ T(c)\|\delta_c\|_2^2}{\sqrt{c_{\infty}}}\bigg]^2
 - 4\bigg[\frac{T(a)\|\delta_a\|_2^2}{\sqrt{a_{\infty}}}(1+\mu_b)\bigg]^2 - 4\bigg[\frac{T(b)\|\delta_b\|_2^2}{\sqrt{b_{\infty}}}(1+\mu_a)\bigg]^2\bigg\}.
    \end{split}
\end{equation}
Since $\|\delta_a\|_2^2 = \overline{a} - \big(\overline{\sqrt{a}}\big)^2$, we get $\|\delta_a\|_2^2 \leq k$; similarly, $\|\delta_b\|_2^2,\|\delta_c\|_2^2 \leq k$. Combined with $T(a),T(b),$ $T(c) \leq \frac{1}{\epsilon}$ and $0\leq 1+ \mu_a,\ 1+ \mu_b,\ 1+ \mu_c\leq 1+\mu_k$, \eqref{Bigg-Ineq} gives
\begin{equation*}
    \begin{split}
& \bigg(\frac{\overline{\sqrt{a}}\overline{\sqrt{b}}^2}{\sqrt{a_{\infty}b_{\infty}^2}} - \frac{\overline{\sqrt{b}}\overline{\sqrt{c}}}{\sqrt{b_{\infty}c_{\infty}}}\bigg)^{2} \geq \frac{\epsilon^4}{b_{\infty}k}
\bigg\{\frac{1}{2}\big[(1+\mu_a)(1+\mu_b)-(1+\mu_c)\big]^2
\\& -\frac{4k^3}{\epsilon^4a_{\infty}}\|\delta_a\|_2^2 - \frac{4k}{\epsilon^2c_{\infty}}\|\delta_c\|_2^2  - \frac{4k(1+\mu_k)^2}{\epsilon^2a_{\infty}}\|\delta_a\|_2^2 - \frac{4k(1+\mu_k)^2}{\epsilon^2b_{\infty}}\|\delta_b\|_2^2\bigg\}
\\& \geq \frac{\epsilon^4}{2b_{\infty}k}\big[(1+\mu_a)(1+\mu_b)-(1+\mu_c)\big]^2 - C_{11}\big(\|\delta_a\|_2^2+\|\delta_b\|_2^2+\|\delta_c\|_2^2\big),
\end{split}
\end{equation*}
where $$C_{11} := \max\bigg\{\frac{4k^2 + 4(1+\mu_k)^2}{a_{\infty}b_{\infty}},\frac{ 4(1+\mu_k)^2}{b_{\infty}^2}, \frac{ 4(1+\mu_k)^2}{b_{\infty}c_{\infty}}\bigg\}.$$
Poincar\'{e}'s inequality reveals 
\begin{equation*}
\begin{split}
& \|\nabla{\sqrt{a}}\|_2^{2}+\|\nabla{\sqrt{b}}\|_2^{2}+\|\nabla{\sqrt{c}}\|_2^{2}+\bigg(\frac{\overline{\sqrt{a}}\overline{\sqrt{b}}^2}{\sqrt{a_{\infty}b_{\infty}^2}} - \frac{\overline{\sqrt{b}}\overline{\sqrt{c}}}{\sqrt{b_{\infty}c_{\infty}}}\bigg)^{2} 
\\& \geq C_{12}\big[(1+\mu_a)(1+\mu_b)-(1+\mu_c)\big]^2,
\end{split}
\end{equation*}
for $C_{12} := \frac{C_P\epsilon^4}{2b_{\infty}kC_{11}}$.
On the other hand,
$$
(\sqrt{\overline{a}}-\sqrt{a_{\infty}})^2 + (\sqrt{\overline{b}}-\sqrt{b_{\infty}})^2 + (\sqrt{\overline{c}}-\sqrt{c_{\infty}})^2 = 
a_{\infty}\mu_a^2+b_{\infty}\mu_b^2+c_{\infty}\mu_c^2,
$$
so we would like to compare $[(1+\mu_a)(1+\mu_b)-(1+\mu_c)]^2$ and $\mu_a^2+\mu_b^2+\mu_c^2$. 
The conservation laws \eqref{conservation-laws} (applied to $\overline{a},\ \overline{b},\ \overline{c}$ and $a_\infty$, $b_\infty$, $c_\infty$) yield 
\begin{equation*}
\begin{split}
& \frac{\overline{a}}{a_{\infty}}=(1+\mu_a), \ \frac{\overline{b}}{b_{\infty}}=(1+\mu_b), \ \frac{\overline{c}}{c_{\infty}}=(1+\mu_c),
\\& \overline{a}+\overline{c} = a_{\infty}+ c_{\infty} = M_1, \ \overline{b}+\overline{c} = b_{\infty}+ c_{\infty} = M_2,
\end{split}
\end{equation*}
and so,
$a_\infty\mu_a+c_\infty\mu_c=b_\infty\mu_b+c_\infty\mu_c=0$.
Thus, unless $\mu_a,\mu_b,\mu_c$ are all zero (trivial case!), we get $\mu_a\mu_c<0$ and $\mu_b\mu_c<0$. If $\mu_a,\mu_b >0$ and $0 >\mu_c$, then \begin{equation*}
\begin{split}
& [(1+\mu_a)(1+\mu_b)-(1+\mu_c)]^2 = (\mu_a\mu_b+\mu_a+\mu_b- \mu_c)^2 
\\& \geq (\mu_a\mu_b+\mu_a+\mu_b)^2 + (\mu_c)^2 > \mu_a^2+\mu_b^2+\mu_c^2.
\end{split}
\end{equation*}
Otherwise, if $\mu_a,\mu_b <0$ and $0 <\mu_c$, then 
\begin{equation*}
\begin{split}
& [(1+\mu_a)(1+\mu_b)-(1+\mu_c)]^2 = (\mu_a\mu_b+\mu_a+\mu_b- \mu_c)^2 
\\& \geq (\mu_a+\mu_b-\mu_c)^2 + (\mu_a\mu_b)^2 > \mu_a^2+\mu_b^2+\mu_c^2.
\end{split}
\end{equation*}
Therefore, in both cases we have  
$$
[(1+\mu_a)(1+\mu_b)-(1+\mu_c)]^2 \geq \mu_a^2+\mu_b^2+\mu_c^2.
$$
(Notice that when $\mu_a=\mu_b=\mu_c=0$, both sides of the inequality are equal to zero.)
Set 
$$C_{10} :=  \frac{C_{12}}{\max(a_{\infty},b_{\infty},c_{\infty})}$$
to conclude the proof of \eqref{C10}.\\

\noindent {\bf Proof of Theorem \ref{conv-theorem1}:}

\begin{proof} Finally, by \eqref{C7}, \eqref{C8}, \eqref{C9} and \eqref{C10}, we obtain
\begin{equation*} 
\begin{split}
& D(a,b,c|a_{\infty},b_{\infty},c_{\infty})
\\& \geq C_7C_9C_{10}\big[\big(\sqrt{\overline{a}}-\sqrt{a_{\infty}}\big)^2 + \big(\sqrt{\overline{b}}-\sqrt{b_{\infty}}\big)^2 + \big(\sqrt{\overline{c}}-\sqrt{c_{\infty}}\big)^2 \big]
\\& \geq \frac{C_7C_9C_{10}}{C_8}E(\overline{a},\overline{b},\overline{c}|a_{\infty},b_{\infty},c_{\infty}).
\end{split}
\end{equation*} 
In view of the above inequality and \eqref{LSI-1}, we discover 
$$
D(a,b,c|a_{\infty},b_{\infty},c_{\infty}) \geq C_{13}E(a,b,c|a_{\infty},b_{\infty},c_{\infty}),
$$
where 
$$C_{13} := \min\bigg(\frac{C_7C_9C_{10}}{C_8}, C_2\bigg).$$

In conclusion, we have proved that the solution decays exponentially to the positive equilibrium (with explicit rate).
\end{proof}

\section{Asymptotic decay for the two-species system}\label{2x2-Syst}
In this section we prove Theorem \ref{conv-theorem2}.
\subsection{Uniform boundedness and global existence for the two-species system}\label{2x2system}

To show the uniform boundedness for classical solutions to \eqref{2x2-system}
we estimate the $L^p$ norm and pass to the limit as $p\rightarrow\infty$. For the lower bound we have
\begin{equation*}
\begin{split}
& \dv{}{t}\int_{\Omega}{a^{-p}dx} = \int_{\Omega}{-pa^{-p-1}a_{t}dx} 
\\& = \int_{\Omega}{-pa^{-p-1}[d_a \Delta a + \bar m(a^{m_2}b^{n_2}-a^{m_1}b^{n_1})]dx}
\\& = \int_{\Omega}{\big[-p(p-1)a^{-p-2}d_a|\nabla{a}|^2 -\bar m pa^{-p-1} (a^{m_2}b^{n_2}-a^{m_1}b^{n_1})\big]dx}
\\& = \int_{\Omega}\big[{-p(p-1)a^{-p-2}d_a|\nabla{a}|^2 -\bar m pa^{-p-1} a^{m_2}b^{n_1}(b^{\bar n}-a^{\bar m})\big] dx}.
    \end{split}
\end{equation*}
Similarly, we have 
\begin{equation*}
\begin{split}
& \dv{}{t}\int_{\Omega}{b^{-q}dx} 
\\& = \int_{\Omega}{\big[-q(q-1)b^{-q-2}d_b|\nabla{b}|^2 -\bar n qb^{-q-1} (a^{m_1}b^{n_1}-a^{m_2}b^{n_2})\big]dx}
\\& = \int_{\Omega}{\big[-q(q-1)b^{-q-2}d_b|\nabla{b}|^2 -\bar n qb^{-q-1}a^{m_2}b^{n_1} (a^{\bar m}-b^{\bar{n}})\big] dx}.
    \end{split}
\end{equation*}
Therefore we can let $p > 1$ be sufficiently large so that $q = \frac{(p+1)\bar{n}}{\bar m}-1 > 1$; then 
\begin{equation*}
\begin{split}
& \dv{}{t}\int_{\Omega}{\bigg(\frac{a^{-p}}{p\bar m} + \frac{b^{-q}}{q\bar{n}}\bigg)dx} 
\\& = \int_{\Omega}{\bigg[-\frac{(p+1)}{\bar m}a^{-p-2}d_a|\nabla{a}|^2 -a^{-p-1} a^{m_2}b^{n_1}(b^{\bar{n}}-a^{\bar m})\bigg] dx} 
\\& + \int_{\Omega}{\bigg[-\frac{(q+1)}{\bar{n}}b^{-q-2}d_b|\nabla{b}|^2 -b^{-q-1}a^{m_2}b^{n_1} (a^{\bar m}-b^{\bar{n}})\bigg]dx}
\\& = \int_{\Omega}{\bigg[-\frac{(p+1)}{\bar m}a^{-p-2}d_a|\nabla{a}|^2 -\frac{(q+1)}{\bar{n}}b^{-q-2}d_b|\nabla{b}|^2\bigg] dx} 
\\& + \int_{\Omega}{(b^{\bar{n}}-a^{\bar m})(b^{-q-1} - a^{-p-1})a^{m_2}b^{n_1} dx} \leq 0.
    \end{split}
\end{equation*}
This gives us
$$
 \int_{\Omega}{\frac{a^{-p}}{p\bar m} dx} \leq \int_{\Omega}{\bigg(\frac{a_{0}^{-p}}{p\bar m} + \frac{b_{0}^{-q}}{q\bar{n}}\bigg) dx} \leq \frac{\alpha^{-p}}{p\bar m} + \frac{\alpha^{-q}}{q\bar{n}}
$$
$$
\implies \int_{\Omega}{a^{-p} dx} \leq \alpha^{-p} + \frac{\alpha^{-q}p\bar m}{q\bar{n}} \leq \alpha^{-p} + \alpha^{-q}\frac{{\bar m}^2}{{\bar{n}}^2} \leq \alpha^{-p} + \alpha^{-q}.
$$
Since 
$$
q+1 = \frac{(p+1)\bar{n}}{\bar m} \implies \frac{q}{p} < \frac{(p+1)\bar{n}}{p\bar m} < 2\frac{\bar{n}}{\bar m},
$$
as we we let $p \rightarrow \infty$, we have
$$
\bigg\|\frac{1}{a}\bigg\|_p < L_a := \max(2,2\alpha^{-2\frac{\bar{n}}{\bar m}}) \implies \bigg\|\frac{1}{a}\bigg\|_{\infty} < L_a.
$$
Also we have a similar result on b, i.e.
$$
 \int_{\Omega}{\frac{b^{-q}}{q\bar{n}} dx} \leq \int_{\Omega}{\bigg(\frac{a_{0}^{-p}}{p\bar{m}} + \frac{b_{0}^{-q}}{q\bar{n}}\bigg) dx} \leq \frac{\alpha^{-p}}{p\bar{m}} + \frac{\alpha^{-q}}{q\bar{n}}
$$
$$
\implies \int_{\Omega}{b^{-q} dx} \leq \frac{\alpha^{-p}q\bar{n}}{p\bar{m}} + \alpha^{-q} \leq 2\alpha^{-p}{\bigg(\frac{\bar{n}}{\bar{m}}\bigg)}^2 + \alpha^{-q} .
$$
Again, as $p \rightarrow \infty$, we have

$$
\bigg\|\frac{1}{b}\bigg\|_q < L_b :=\bigg[2{\bigg(\frac{\bar{n}}{\bar{m}}\bigg)}^2+ 1\bigg] \cdot \max\bigg(1,\frac{1}{\alpha}\bigg) \implies \bigg\|\frac{1}{b}\bigg\|_{\infty} < L_b.
$$
Thus, we get the uniform lower bound $\epsilon^2$ for $a(x,t)$ and $b(x,t)$ with $\epsilon^2 := \min(L_a^{-1},L_b^{-1})$. As for the upper bound, we use the same method to see that
\begin{equation*}
 \dv{}{t}\int_{\Omega}{a^{p}dx}  = \int_{\Omega}{\big[-p(p-1)a^{p-2}d_a|\nabla{a}|^2 +\bar m pa^{p-1} a^{m_2}b^{n_1}(b^{\bar{n}}-a^{\bar{m}})\big] dx}.
\end{equation*}
Similarly, we have
\begin{equation*}
\begin{split}
& \dv{}{t}\int_{\Omega}{b^{q}dx} 
 = \int_{\Omega}{\big[-q(q-1)b^{q-2}d_b|\nabla{b}|^2 +\bar n qb^{q-1}a^{m_2}b^{n_1} (a^{\bar{m}}-b^{\bar{n}})\big] dx}.
    \end{split}
\end{equation*}
Therefore, we can let $p > 1$ be sufficient large to render $q = \frac{(p-1)\bar{n}}{\bar{m}}+1 > 1$; then 
\begin{equation*}
\begin{split}
& \dv{}{t}\int_{\Omega}{\bigg(\frac{a^{p}}{p\bar{m}} + \frac{b^{q}}{q\bar{n}}\bigg)dx} 
\\& = \int_{\Omega}{\bigg[-\frac{(p-1)}{\bar{m}}a^{p-2}d_a|\nabla{a}|^2 -\frac{(q-1)}{\bar{n}}b^{q-2}d_b|\nabla{b}|^2\bigg] dx} 
\\& - \int_{\Omega}{(b^{\bar{n}}-a^{\bar{m}})(b^{q-1} - a^{p-1})a^{m_2}b^{n_1} dx} \leq 0.
    \end{split}
\end{equation*}
This gives us
$$
 \int_{\Omega}{\frac{a^{p}}{p\bar{m}} dx} \leq \int_{\Omega}{\bigg(\frac{a_{0}^{p}}{p\bar{m}} + \frac{b_{0}^{q}}{q\bar{n}}\bigg) dx} \leq \frac{\beta^{p}}{p\bar{m}} + \frac{\beta^{q}}{q\bar{n}}
$$
$$
\implies \int_{\Omega}{a^{p} dx} \leq \beta^{p} + \beta^{q}\frac{p\bar{m}}{q\bar{n}} \leq \beta^{p} + \beta^{q}.
$$
But
$$
q-1 = \frac{(p-1)\bar{n}}{\bar{m}} \implies \frac{q}{p} < \frac{q-1}{p-1} = \frac{\bar{n}}{\bar{m}},
$$
so, when we let $p \rightarrow \infty$, we obtain
$$
\|a\|_p < U_a := \max(2,2\beta^{\frac{\bar{n}}{\bar{m}}}) \implies \|a\|_{\infty} < U_a.
$$
We have a similar estimate on $b$,
$$
 \int_{\Omega}{\frac{b^{q}}{q\bar{n}} dx} \leq \int_{\Omega}\bigg({\frac{a_{0}^{p}}{p\bar{m}} + \frac{b_{0}^{q}}{q\bar{n}}\bigg) dx} \leq \frac{\beta^{p}}{p\bar{m}} + \frac{\beta^{q}}{q\bar{n}} 
$$
$$
\implies \int_{\Omega}{b^{q} dx} \leq \frac{\beta^{p}q\bar{n}}{p\bar{m}} + \beta^{q} \leq \beta^{p}{\bigg(\frac{\bar{n}}{\bar{m}}\bigg)}^2 + \beta^{q} .
$$
Again, as $p \rightarrow \infty$, we get
$$
\|b\|_q < U_b = \bigg[{\bigg(\frac{\bar{n}}{\bar{m}}\bigg)}^2+ 1\bigg] \cdot \max(1,\beta) \implies \|b\|_{\infty} < U_b.
$$
Thus, we get the uniform upper bound $\omega$ for $a(x,t)$ and $b(x,t)$ with $\omega := \max(U_a,U_b)$, and this also implies the existence of a unique global classical solution.

\subsection{Convergence for the two species system}

We use the same entropy entropy dissipation method to obtain an explicit exponential convergence rate for the two species system in any dimension.

Again we introduce the  relative entropy 
\begin{equation*} 
\begin{split}
& E(a,b|a_{\infty},b_{\infty}) =  \int_{\Omega}\bigg({a\ln{\frac{a}{a_{\infty}}}}-a+a_{\infty}\bigg)dx+\int_{\Omega}\bigg({b\ln{\frac{b}{b_{\infty}}}}-b+b_{\infty}\bigg)dx
\end{split} 
\end{equation*}
and its corresponding entropy dissipation
\begin{equation*} 
\begin{split}
& D(a,b|a_{\infty},b_{\infty}) = d_a\int_{\Omega}{\frac{|\nabla{a}|^2}{a}}dx+d_b\int_{\Omega}{\frac{|\nabla{b}|^2}{b}}dx
\\& + a_{\infty}^{m_1}b_{\infty}^{n_1}\int_{\Omega}{\Psi{\bigg(\frac{a^{m_1}b^{n_1}}{a_{\infty}^{m_1}b_{\infty}^{n_1}} ; \frac{a^{m_2}b^{n_2}}{a_{\infty}^{m_2}b_{\infty}^{n_2}}\bigg)}}dx + a_{\infty}^{m_2}b_{\infty}^{n_2}\int_{\Omega}{\Psi{\bigg( \frac{a^{m_2}b^{n_2}}{a_{\infty}^{m_2}b_{\infty}^{n_2}} ; \frac{a^{m_1}b^{n_1}}{a_{\infty}^{m_1}b_{\infty}^{n_1}}\bigg)}}dx.
\end{split} 
\end{equation*}
Due to the following identity 
$$
E(a,b|a_{\infty},b_{\infty}) = E(a,b|\overline{a},\overline{b}) + E(\overline{a},\overline{b}|a_{\infty},b_{\infty})
$$
and the Logarithmic Sobolev Inequality \eqref{LSI} we have
\begin{equation}\label{D1}
d_a\int_{\Omega}{\frac{|\nabla{a}|^2}{a}}dx+d_b\int_{\Omega}{\frac{|\nabla{b}|^2}{b}}dx \geq 
D_1  E(a,b|\overline{a},\overline{b}),
\end{equation}
where $D_1 = \min(d_a,d_b) \cdot C_{LSI}$.

From inequality \eqref{Psi-ineq} we get the following estimate
\begin{equation} \label{D2}
\begin{split}
& D(a,b|a_{\infty},b_{\infty}) = d_a\int_{\Omega}{\frac{|\nabla{a}|^2}{a}}dx+d_b\int_{\Omega}{\frac{|\nabla{b}|^2}{b}}dx
\\& + a_{\infty}^{m_1}b_{\infty}^{n_1}\int_{\Omega}{\Psi{\bigg(\frac{a^{m_1}b^{n_1}}{a_{\infty}^{m_1}b_{\infty}^{n_1}} ; \frac{a^{m_2}b^{n_2}}{a_{\infty}^{m_2}b_{\infty}^{n_2}}\bigg)}}dx + a_{\infty}^{m_2}b_{\infty}^{n_2}\int_{\Omega}{\Psi{\bigg( \frac{a^{m_2}b^{n_2}}{a_{\infty}^{m_2}b_{\infty}^{n_2}} ; \frac{a^{m_1}b^{n_1}}{a_{\infty}^{m_1}b_{\infty}^{n_1}}\bigg)}}dx
\\& \geq 4d_a\|\nabla{\sqrt{a}}\|_2^{2}+4d_b\|\nabla{\sqrt{b}}\|_2^{2}+4d_c\|\nabla{\sqrt{c}}\|_2^{2}
\\& + a_{\infty}^{m_1}b_{\infty}^{n_1}\bigg\|\sqrt{\frac{a^{m_1}b^{n_1}}{a_{\infty}^{m_1}b_{\infty}^{n_1}}}   - \sqrt{\frac{a^{m_2}b^{n_2}}{a_{\infty}^{m_2}b_{\infty}^{n_2}}}\bigg\|_2^{2} + a_{\infty}^{m_2}b_{\infty}^{n_2}\bigg\| \sqrt{\frac{a^{m_2}b^{n_2}}{a_{\infty}^{m_2}b_{\infty}^{n_2}}} - \sqrt{\frac{a^{m_1}b^{n_1}}{a_{\infty}^{m_1}b_{\infty}^{n_1}}}\bigg\|_2^{2}
\\& \geq D_2\bigg(\|\nabla{\sqrt{a}}\|_2^{2}+\|\nabla{\sqrt{b}}\|_2^{2}+\|\nabla{\sqrt{c}}\|_2^{2}+\bigg\|\sqrt{\frac{a^{m_1}b^{n_1}}{a_{\infty}^{m_1}b_{\infty}^{n_1}}}   - \sqrt{\frac{a^{m_2}b^{n_2}}{a_{\infty}^{m_2}b_{\infty}^{n_2}}}\bigg\|_2^{2}\bigg),
\end{split} 
\end{equation}
where $D_2 = \min(4d_a,4d_b,4d_c,a_{\infty}^{m_1}b_{\infty}^{n_1}+a_{\infty}^{m_2}b_{\infty}^{n_2})$.

The obvious conservation law
$$\bar n\overline{a(t)}+\bar m\overline{b(t)}=\bar n\overline{a_0}+\bar m\overline{b_0}\mbox{ for all }t>0$$ shows there exists $N$ such that $\overline{a},\overline{b} < N$.
Thus, \eqref{psi-increase} yields
$$
 \Psi(x,y) \leq \frac{\Psi(N,y)}{(\sqrt{N}-\sqrt{y})^2}(\sqrt{x}-\sqrt{y})^2\mbox{ for all }x \leq N.
$$
Since $0 < a_{\infty},b_{\infty} < N$,
\begin{equation} \label{D3}
\begin{split}
& E(\overline{a},\overline{b}|a_{\infty},b_{\infty}) = \bigg({\overline{a}\ln{\frac{\overline{a}}{a_{\infty}}}}-\overline{a}+a_{\infty}\bigg) + \bigg({\overline{b}\ln{\frac{\overline{b}}{b_{\infty}}}}-\overline{b}+b_{\infty}\bigg) 
\\&  < \frac{\Psi(N,a_{\infty})}{(\sqrt{N}-\sqrt{a_{\infty}})^2}(\sqrt{\overline{a}}-\sqrt{a_{\infty}})^2 
+ \frac{\Psi(N,b_{\infty})}{(\sqrt{N}-\sqrt{b_{\infty}})^2}(\sqrt{\overline{b}}-\sqrt{b_{\infty}})^2 
\\& \leq D_3\big[\big(\sqrt{\overline{a}}-\sqrt{a_{\infty}}\big)^2 + \big(\sqrt{\overline{b}}-\sqrt{b_{\infty}}\big)^2 \big],
\end{split} 
\end{equation}
where 
$$D_3 = \max\bigg\{\frac{\Psi(N,a_{\infty})}{(\sqrt{N}-\sqrt{a_{\infty}})^2},\frac{\Psi(N,b_{\infty})}{(\sqrt{N}-\sqrt{b_{\infty}})^2}\bigg\}.$$

Now we claim there exists a constant $D_4 > 0$ such that
\begin{equation}\label{D4}
\begin{split}
& \|\nabla{\sqrt{a}}\|_2^{2}+\|\nabla{\sqrt{b}}\|_2^{2}+
\bigg\|\sqrt{\frac{a^{m_1}b^{n_1}}{a_{\infty}^{m_1}b_{\infty}^{n_1}}}   - \sqrt{\frac{a^{m_2}b^{n_2}}{a_{\infty}^{m_2}b_{\infty}^{n_2}}}\bigg\|_2^{2} 
\\& > D_4 \bigg\{ \|\nabla{\sqrt{a}}\|_2^{2}+\|\nabla{\sqrt{b}}\|_2^{2}+
\bigg(\frac{{\overline{\sqrt{a}}}^{m_1}{\overline{\sqrt{b}}}^{n_1}}{\sqrt{a_{\infty}^{m_1}b_{\infty}^{n_1}}} - \frac{\overline{\sqrt{a}}^{m_2} \overline{\sqrt{b}}^{n_2}}{\sqrt{a_{\infty}^{m_2}b_{\infty}^{n_2}}}\bigg)^{2} \bigg\} . 
\end{split}
\end{equation}

Now we again introduce the deviations $\delta_a = \sqrt{a}-\overline{\sqrt{a}},\delta_b = \sqrt{b}-\overline{\sqrt{b}}$ and 
make the decomposition 
$$
\Omega = D_L \cup D_L^{\complement},
$$
where $D_L := \{ x \in \Omega\,:|\,\delta_a|,|\delta_b| \leq L \}$ with a fixed constant $L$. On the set $D_L$ we get
\begin{equation*}
\begin{split}
\sqrt{a^{m_1}b^{n_1}}  &= (\overline{\sqrt{a}}+\delta_a)^{m_1}(\overline{\sqrt{b}}+\delta_b)^{n_1} 
\\& \leq \overline{\sqrt{a}}^{m_1} \cdot \overline{\sqrt{b}}^{n_1} +  (|\delta_a|+|\delta_b|)R_1(|\delta_a|,|\delta_b|,\overline{\sqrt{a}},\overline{\sqrt{b}}),
\end{split}
\end{equation*}

\begin{equation*}
\begin{split}
\sqrt{a^{m_2}b^{n_2}} &= (\overline{\sqrt{a}}+\delta_a)^{m_2}(\overline{\sqrt{b}}+\delta_b)^{n_2} 
\\& \leq \overline{\sqrt{a}}^{m_2} \cdot \overline{\sqrt{b}}^{n_2} + (|\delta_a|+|\delta_b|)R_2(|\delta_a|,|\delta_b|,\overline{\sqrt{a}},\overline{\sqrt{b}}),
\end{split}
\end{equation*}
where $R_1 $ and $R_2$ are finite due to the boundedness of $|\delta_a|,|\delta_b|,\overline{\sqrt{a}},\overline{\sqrt{b}}$.
Then we get 
\begin{equation*}
\begin{split}
& \bigg\|\sqrt{\frac{a^{m_1}b^{n_1}}{a_{\infty}^{m_1}b_{\infty}^{n_1}}}   - \sqrt{\frac{a^{m_2}b^{n_2}}{a_{\infty}^{m_2}b_{\infty}^{n_2}}}\bigg\|_{L^{2}(D_L)}^{2}
\\& \geq \frac{1}{2}\bigg(\frac{{\overline{\sqrt{a}}}^{m_1}{\overline{\sqrt{b}}}^{n_1}}{\sqrt{a_{\infty}^{m_1}b_{\infty}^{n_1}}} - \frac{\overline{\sqrt{a}}^{m_2} \overline{\sqrt{b}}^{n_2}}{\sqrt{a_{\infty}^{m_2}b_{\infty}^{n_2}}}\bigg)^2|D_L| - 2\|(|\delta_a|+|\delta_b|)\|_{L^{2}(D_L)}^{2}\frac{R_1^2}{a_{\infty}^{m_1}b_{\infty}^{n_1}} 
\\& - 2\|(|\delta_a|+|\delta_b|)\|_{L^{2}(D_L)}^{2}\frac{R_2^{2}}{a_{\infty}^{m_2}b_{\infty}^{n_2}}
\\& \geq \frac{1}{2}\bigg(\frac{{\overline{\sqrt{a}}}^{m_1}{\overline{\sqrt{b}}}^{n_1}}{\sqrt{a_{\infty}^{m_1}b_{\infty}^{n_1}}} - \frac{\overline{\sqrt{a}}^{m_2} \overline{\sqrt{b}}^{n_2}}{\sqrt{a_{\infty}^{m_2}b_{\infty}^{n_2}}}\bigg)^2|D_L| - R(|\delta_a|,|\delta_b|,\overline{\sqrt{a}},\overline{\sqrt{b}})[\|\delta_a\|_{L^{2}(D_L)}^{2}+
\\& \|\delta_b\|_{L^{2}(D_L)}^{2}],
\end{split} 
\end{equation*}
where $R(|\delta_a|,|\delta_b|,\overline{\sqrt{a}},\overline{\sqrt{b}}) = \frac{4R_1^2}{a_{\infty}^{m_1}b_{\infty}^{n_1}}  + \frac{4R_2^{2}}{a_{\infty}^{m_2}b_{\infty}^{n_2}}$ is finite  (depends on the choice of $L$ and $N$).

On the set $D_L^{\complement}$, by using Poincar\'{e}'s inequality, we get 
\begin{equation*}
\begin{split}
& \|\nabla{\sqrt{a}}\|_2^{2}+\|\nabla{\sqrt{b}}\|_2^{2}
\geq C_P \big(\|\delta_a\|_{L^{2}(D_L^{\complement})}^{2}+ \|\delta_b\|_{L^{2}(D_L^{\complement})}^{2}\big)
\geq C_P L^2 |D_L^{\complement}| .
\end{split}
\end{equation*}
Since $$ \bigg|\frac{{\overline{\sqrt{a}}}^{m_1}{\overline{\sqrt{b}}}^{n_1}}{\sqrt{a_{\infty}^{m_1}b_{\infty}^{n_1}}} - \frac{\overline{\sqrt{a}}^{m_2} \overline{\sqrt{b}}^{n_2}}{\sqrt{a_{\infty}^{m_2}b_{\infty}^{n_2}}}\bigg| \leq\frac{\sqrt{N}^{m_1+n_1}}{\sqrt{a_{\infty}^{m_1}b_{\infty}^{n_1}}}  +  \frac{\sqrt{N}^{m_2+n_2}}{a_{\infty}^{m_2}b_{\infty}^{n_2}},$$ we infer 
$$
\|\nabla{\sqrt{a}}\|_2^{2}+\|\nabla{\sqrt{b}}\|_2^{2} \geq \Tilde{R} \bigg(\frac{{\overline{\sqrt{a}}}^{m_1}{\overline{\sqrt{b}}}^{n_1}}{\sqrt{a_{\infty}^{m_1}b_{\infty}^{n_1}}} - \frac{\overline{\sqrt{a}}^{m_2} \overline{\sqrt{b}}^{n_2}}{\sqrt{a_{\infty}^{m_2}b_{\infty}^{n_2}}}\bigg)^2|D_L^{\complement}|,
$$
where 
$$\Tilde{R} = C_P L^2{\bigg(\frac{\sqrt{N}^{m_1+n_1}}{\sqrt{a_{\infty}^{m_1}b_{\infty}^{n_1}}}  +  \frac{\sqrt{N}^{m_2+n_2}}{a_{\infty}^{m_2}b_{\infty}^{n_2}}\bigg)^{-2}}.$$
We combine the above two parts, pick $K > \frac{R+1}{\min\{1,C_P\}}$ and have the following
\begin{equation*}
\begin{split}
& 3K(\|\nabla{\sqrt{a}}\|_2^{2}+\|\nabla{\sqrt{b}}\|_2^{2}) +  \bigg\|\sqrt{\frac{a^{m_1}b^{n_1}}{a_{\infty}^{m_1}b_{\infty}^{n_1}}}   - \sqrt{\frac{a^{m_2}b^{n_2}}{a_{\infty}^{m_2}b_{\infty}^{n_2}}}\bigg\|_2^{2}  
\\& \geq K(\|\nabla{\sqrt{a}}\|_2^{2}+\|\nabla{\sqrt{b}}\|_2^{2}) + K\Tilde{R} \bigg(\frac{{\overline{\sqrt{a}}}^{m_1}{\overline{\sqrt{b}}}^{n_1}}{\sqrt{a_{\infty}^{m_1}b_{\infty}^{n_1}}} - \frac{\overline{\sqrt{a}}^{m_2} \overline{\sqrt{b}}^{n_2}}{\sqrt{a_{\infty}^{m_2}b_{\infty}^{n_2}}}\bigg)^2|D_L^{\complement}| 
\\& +\bigg\{ \frac{1}{2}\bigg(\frac{{\overline{\sqrt{a}}}^{m_1}{\overline{\sqrt{b}}}^{n_1}}{\sqrt{a_{\infty}^{m_1}b_{\infty}^{n_1}}} - \frac{\overline{\sqrt{a}}^{m_2} \overline{\sqrt{b}}^{n_2}}{\sqrt{a_{\infty}^{m_2}b_{\infty}^{n_2}}}\bigg)^2|D_L| - R\big[\|\delta_a\|_{L^{2}(D_L)}^{2}+ \|\delta_b\|_{L^{2}(D_L)}^{2}\big] \bigg\}
\\&  + KC_P(\|\delta_a\|_{L^{2}(D_L)}^{2}+ \|\delta_b\|_{L^{2}(D_L)}^{2})
\\& \geq C_{K,R}\bigg[ \|\nabla{\sqrt{a}}\|_2^{2}+\|\nabla{\sqrt{b}}\|_2^{2}+\bigg(\frac{{\overline{\sqrt{a}}}^{m_1}{\overline{\sqrt{b}}}^{n_1}}{\sqrt{a_{\infty}^{m_1}b_{\infty}^{n_1}}} - \frac{\overline{\sqrt{a}}^{m_2} \overline{\sqrt{b}}^{n_2}}{\sqrt{a_{\infty}^{m_2}b_{\infty}^{n_2}}}\bigg)^2\bigg], 
\end{split}
\end{equation*}
where $ C_{K,R} := \min\{K\Tilde{R},\frac{1}{2}\}$.
We can fix $L >0$ and get the corresponding $R$, then pick sufficiently large $K$ (e.g. $K > R+1$) such that we obtain \eqref{D4} with $D_4 = \frac{C_{K,R}}{3K}$.

It remains to show that there exists a constant $D_{5}$ such that
\begin{equation}\label{D5}
\begin{split}
 & \|\nabla{\sqrt{a}}\|_2^{2}+\|\nabla{\sqrt{b}}\|_2^{2}+\bigg(\frac{{\overline{\sqrt{a}}}^{m_1}{\overline{\sqrt{b}}}^{n_1}}{\sqrt{a_{\infty}^{m_1}b_{\infty}^{n_1}}} - \frac{\overline{\sqrt{a}}^{m_2} \overline{\sqrt{b}}^{n_2}}{\sqrt{a_{\infty}^{m_2}b_{\infty}^{n_2}}}\bigg)^2 
 \\& > D_{5}\big[\big(\sqrt{\overline{a}}-\sqrt{a_{\infty}}\big)^2 + \big(\sqrt{\overline{b}}-\sqrt{b_{\infty}}\big)^2\big].
\end{split} 
\end{equation}

We again introduce $\mu_a,\mu_b$ to parameterize
$\sqrt{\overline{a}},\sqrt{\overline{b}}$ with 
$\sqrt{\overline{a}}=\sqrt{a_{\infty}}(1+\mu_a)$,$\sqrt{\overline{b}}=\sqrt{b_{\infty}}(1+\mu_b)$, where $\mu_{\epsilon} \leq \mu_a,\mu_b < \mu_\omega$ with $\mu_{\epsilon} = \frac{\epsilon}{\max\{\sqrt{a_{\infty}},\sqrt{b_{\infty}}\}}-1$ and $\mu_\omega = \frac{\sqrt{\omega}}{\min\{\sqrt{a_{\infty}},\sqrt{b_{\infty}}\}}-1$.
We have 
$ 
\overline{\sqrt{a}} = -\frac{\|\delta_a\|_2^2}{\sqrt{\overline{a}} + \overline{\sqrt{a}}} + \sqrt{\overline{a}} = \sqrt{\overline{a}} - T(a)\|\delta_a\|_2^2
$
, where $T(a) = \frac{1}{\sqrt{\overline{a}}+\overline{\sqrt{a}}} $.
Similarly, 
$
\overline{\sqrt{b}}= \sqrt{\overline{b}} - T(b)\|\delta_b\|_2^2 
$
, where $T(b) = \frac{1}{\sqrt{\overline{b}} + \overline{\sqrt{b}}}$. Both $T(a), T(b)$ have uniform (in time) upper and lower bounds. Since we get the lower bound for $\overline{\sqrt{a}},\overline{\sqrt{b}}$, we can factor out 
$\bigg(\frac{{\overline{\sqrt{a}}}^{m_2}{\overline{\sqrt{b}}}^{n_1}}{\sqrt{a_{\infty}}^{m_2}\sqrt{b_{\infty}}^{n_1}}\bigg)^2$ to reduce the original system \eqref{2x2-system} to the system associated with the reaction $\bar m A{\rightleftharpoons}\bar n B$ (which does not  have boundary equilibria).
Then we get 
\begin{equation*}
    \begin{split}
& \bigg(\frac{{\overline{\sqrt{a}}}^{m_1}{\overline{\sqrt{b}}}^{n_1}}{\sqrt{a_{\infty}^{m_1}b_{\infty}^{n_1}}} - \frac{\overline{\sqrt{a}}^{m_2} \overline{\sqrt{b}}^{n_2}}{\sqrt{a_{\infty}^{m_2}b_{\infty}^{n_2}}}\bigg)^2  = \frac{({\overline{\sqrt{a}}}^{m_2}{\overline{\sqrt{b}}}^{n_1})^2}{a_{\infty}^{m_2}b_{\infty}^{n_1}} \cdot \bigg(\frac{{\overline{\sqrt{a}}}^{\bar{m}}}{\sqrt{a_{\infty}}^{\bar{m}}} - \frac{\overline{\sqrt{b}}^{\bar{n}}}{\sqrt{b_{\infty}}^{\bar{n}}}\bigg)^{2} 
\\& \geq \frac{\epsilon^{2(m_2+n_1)}}{a_{\infty}^{m_2}b_{\infty}^{n_1}}
\bigg[\bigg(1+\mu_a - \frac{T(a)\|\delta_a\|_2^2}{\sqrt{a_{\infty}}}\bigg)^{\bar{m}}-\bigg(1+\mu_b - \frac{T(b)\|\delta_b\|_2^2}{\sqrt{b_{\infty}}}\bigg)^{\bar{n}}\bigg]^{2} =:A.
\end{split}
\end{equation*}
We evaluate 
\begin{equation*}
    \begin{split}
& A= \frac{\epsilon^{2(m_2+n_1)}}{a_{\infty}^{m_2}b_{\infty}^{n_1}}
\big\{[(1+\mu_a )^{\bar{m}}+ 
\|\delta_a\|_2 S_1(\mu_a,\|\delta_a\|_2,T(a))]
-[(1+\mu_b )^{\bar{n}}
\\& + 
\|\delta_b\|_2 S_2(\mu_b,\|\delta_b\|_2,T(b))]\big\}^{2} 
\\& \geq \frac{\epsilon^{2(m_2+n_1)}}{a_{\infty}^{m_2}b_{\infty}^{n_1}}
\bigg\{\frac{1}{2}[(1+\mu_a )^{\bar{m}}-(1+\mu_b )^{\bar{n}}]^2-2[\|\delta_a\|_2 S_1-\|\delta_b\|_2 S_2]^2\bigg\}
\\& \geq \frac{\epsilon^{2(m_2+n_1)}}{a_{\infty}^{m_2}b_{\infty}^{n_1}}
\bigg\{\frac{1}{2}[(1+\mu_a )^{\bar{m}}-(1+\mu_b )^{\bar{n}}]^2-4(\|\delta_a\|_2^2 +\|\delta_b\|_2^2)S\bigg\},
\end{split}
\end{equation*}
where $S = \max{(|S_1|,|S_2|)}$.

We have
$
\|\delta_a\|_2^2 = \overline{a} - (\overline{\sqrt{a}})^2 \leq \omega
$, similarly $\|\delta_b\|_2^2 \leq \omega$ and $T(a),T(b) \leq \frac{1}{\epsilon}$,
 $ \mu_a,\mu_b < \mu_\omega$; so $S$ is uniformly bounded and 
\begin{equation*}
    \begin{split}
& \bigg(\frac{{\overline{\sqrt{a}}}^{m_1}{\overline{\sqrt{b}}}^{n_1}}{\sqrt{a_{\infty}^{m_1}b_{\infty}^{n_1}}} - \frac{\overline{\sqrt{a}}^{m_2} \overline{\sqrt{b}}^{n_2}}{\sqrt{a_{\infty}^{m_2}b_{\infty}^{n_2}}}\bigg)^2 
\\& \geq \frac{\epsilon^{2(m_2+n_1)}}{2a_{\infty}^{m_2}b_{\infty}^{n_1}}
[(1+\mu_a )^{\bar{m}}-(1+\mu_b )^{\bar{n}}]^2-D_6(\|\delta_a\|_2^2 +\|\delta_b\|_2^2),
\end{split}
\end{equation*}
where $D_{6} = 4\frac{\epsilon^{2(m_2+n_1)}}{a_{\infty}^{m_2}b_{\infty}^{n_1}}\|S\|_{\infty}$.
Poincar\'{e}'s inequality yields 
\begin{equation*}
\begin{split}
 & \|\nabla{\sqrt{a}}\|_2^{2}+\|\nabla{\sqrt{b}}\|_2^{2}+\bigg(\frac{{\overline{\sqrt{a}}}^{m_1}{\overline{\sqrt{b}}}^{n_1}}{\sqrt{a_{\infty}^{m_1}b_{\infty}^{n_1}}} - \frac{\overline{\sqrt{a}}^{m_2} \overline{\sqrt{b}}^{n_2}}{\sqrt{a_{\infty}^{m_2}b_{\infty}^{n_2}}}\bigg)^2 
 \\& > D_{7}[(1+\mu_a )^{\bar{m}}-(1+\mu_b )^{\bar{n}}]^2,
\end{split} 
\end{equation*}
where $D_{7} = \frac{\epsilon^{2(m_2+n_1)}}{2a_{\infty}^{m_2}b_{\infty}^{n_1}} D_6^{-1}$.
On the other hand,
$$
(\sqrt{\overline{a}}-\sqrt{a_{\infty}})^2 + (\sqrt{\overline{b}}-\sqrt{b_{\infty}})^2 = 
a_{\infty}\mu_a^2+b_{\infty}\mu_b^2,
$$
so we need to compare $[(1+\mu_a )^{\bar{m}}-(1+\mu_b )^{\bar{n}}]^2$ with $\mu_a^2+\mu_b^2$. The conservation law gives
$$
\frac{\overline{a}}{a_{\infty}}=(1+\mu_a), \ \frac{\overline{b}}{b_{\infty}}=(1+\mu_b),
\ \bar n \overline{a}+ \bar{m} \overline{b} = \bar n a_{\infty}+ \bar m b_{\infty},
$$
so, as long as $\mu_a,\mu_b$ are nonzero, $\mu_a$ and $\mu_b$ must have different signs. If $\mu_a >0 >\mu_b$, then 
$$
 [(1+\mu_a )^{\bar m}-(1+\mu_b )^{\bar n}]^2 
 \geq [(1+\mu_a )-(1+\mu_b )]^2 > \mu_a^2+\mu_b^2.
$$
Otherwise, 
$$
 [(1+\mu_a )^{\bar m}-(1+\mu_b )^{\bar n}]^2 
 = [(1+\mu_b )^{\bar n}-(1+\mu_a )^{\bar m}]^2  > \mu_a^2+\mu_b^2,
$$
so in both cases we have  
$$
[(1+\mu_a )^{\bar m}-(1+\mu_b )^{\bar n}]^2 \geq \mu_a^2+\mu_b^2.
$$
Notice that when $\mu_a=\mu_b=0$, both sides of the inequality are equal to zero.\\

\noindent {\bf Proof of Theorem \ref{conv-theorem2}:}

\begin{proof}
Set $D_{5} = \frac{D_{7}}{\max(a_{\infty},b_{\infty})}$ to see that \eqref{D5} holds, and then combine \eqref{D2}, \eqref{D3}, \eqref{D4} and \eqref{D5} to reveal
$$
 D(a,b|a_{\infty},b_{\infty})
 \geq \frac{D_2D_4D_{5}}{D_3}E(\overline{a},\overline{b}|a_{\infty},b_{\infty}).
$$ 
In view of \eqref{D1}, we get 
$$
D(a,b|a_{\infty},b_{\infty}) \geq D_{8}E(a,b|a_{\infty},b_{\infty})
$$
for $D_{8} = \min(\frac{D_2D_4D_{5}}{D_3}, D_1) $, which finally proves that the solution decays exponentially to the positive equilibrium at an explicit rate.
\end{proof}

\section{Remarks and open problems}\label{sec:remarks}
We believe this approach works to prove and quantify the decay rate to the complex balanced equilibrium for more general systems of the type
$$A_1+A_2+...+A_{n-1}+mA_n{\rightleftharpoons}A_n+A_{n+1},$$
and even more complex systems such as
$$A_1+A_2+...+A_{n-1}+mA_n{\rightleftharpoons}A_n+A_{n+1}{\leftrightharpoons}B_1+...+B_k+lA_n,$$
where $k\geq 1$ and $l,\ m\geq 2$ are integers. 

In order to obtain the uniform essential bound on the densities we used the $L^\infty$ version of Poincar\'{e}'s inequality, which is only available in 1D. Can we employ more refined techniques in order to prove the essential bounds in higher dimensions? These are some questions we plan to address in future work.

\section{Appendix}

\begin{lemma}\label{Psi-lemma}
For any $x,\ y>0$ we have 
\begin{equation}\label{Psi-ineq}
\Psi{(x,y)}=x\ln{\frac{x}{y}}-x+y \geq (\sqrt{x}-\sqrt{y})^2.
\end{equation}
\end{lemma}
\proof{} The case $x=y$ is trivial.  If $x>y>0$ we use the Jensen inequality for the convex function $f(s):=[y+s(x-y)]^{-1}$ to get
$$\frac{\ln x-\ln y}{x-y}=\int_0^1f(s)ds\geq f\bigg(\int_0^1sds\bigg)=\frac{2}{x+y}.$$
By using this inequality and $x+y>2\sqrt{xy}$, we conclude
\begin{equation*}
    \begin{split}
& \Psi{(x,y)} \geq x\frac{2(x-y)}{(x+y)}-x+y
 =x\bigg(2-\frac{4y}{x+y}\bigg)-x+y 
 \\& > x\bigg(2-\frac{4y}{2\sqrt{xy}}\bigg)-x+y = x-2\sqrt{xy}+y = (\sqrt{x}-\sqrt{y})^2.
    \end{split}
\end{equation*}
Or suppose we have $y>x> 0$ and set $g(u)=e^u$, $u(s)=a+s(b-a)$, for $b>a$. Jensen's inequality shows
$$\frac{e^b-e^a}{b-a}=\int_0^1g(u(s))ds\geq g\bigg(\int_0^1 u(s)ds\bigg)=e^{\frac{a+b}{2}}.$$
Let $b=\ln y,\ a=\ln x$ to deduce
\begin{equation}\label{ln-sqrt}
\frac{x-y}{\ln{x}-\ln{y}} \geq \sqrt{xy},
\end{equation}
which implies $$ \frac{x-y}{\sqrt{xy}} \leq \ln{\bigg(\frac{x}{y}\bigg)}.$$
It follows
\begin{equation*}
    \begin{split}
& \Psi{(x,y)} \geq x\frac{(x-y)}{\sqrt{xy}}-x+y
 =\sqrt{\frac{x}{y}}x-\sqrt{xy}-x+y
 \\& =\sqrt{\frac{x}{y}}x+\sqrt{xy}-2\sqrt{xy}-x+y \geq 2x-2\sqrt{xy}-x+y = (\sqrt{x}-\sqrt{y})^2.
    \end{split}
\end{equation*}
\endproof

The second important tool is:
\begin{lemma}\label{lemma-2}
 For each fixed $y > 0$,
\begin{equation}\label{psi-increase}
 \psi(x,y) := \frac{\Psi{(x,y)}}{(\sqrt{x}-\sqrt{y})^2}\mbox{ is increasing in }x\in(0,\infty).   
 \end{equation}
 \end{lemma}
\proof{} Suppose $x > y > 0$ and set $k = \sqrt{\frac{x}{y}} > 1$. We have
\begin{equation*}
    \begin{split}
& \dv{\psi{(x,y)}}{x}=\frac{\ln{\big(\frac{x}{y}\big)}(\sqrt{x}-\sqrt{y})^2-\big[x\ln{\big(\frac{x}{y}\big)}-x+y\big]\big(\sqrt{x}-\sqrt{y}\big)\frac{1}{\sqrt{x}}}{(\sqrt{x}-\sqrt{y})^4}
\\& = \frac{\ln{\big(\frac{x}{y}\big)}\big(\sqrt{x}-\sqrt{y}\big)-\big[x\ln{\big(\frac{x}{y}\big)}-x+y\big]\frac{1}{\sqrt{x}}}{(\sqrt{x}-\sqrt{y})^3}
\\& = \frac{\big[2(k-1)\ln k-\big(2k\ln{k}-k+\frac{1}{k}\big)\big]\sqrt{y}}{(\sqrt{x}-\sqrt{y})^3} 
= \frac{\big(k-2\ln{k}-\frac{1}{k}\big)\sqrt{y}}{(\sqrt{x}-\sqrt{y})^3}.
    \end{split}
\end{equation*}
Since we have $k>1$, we use \eqref{ln-sqrt} to obtain
$$
\frac{k-\frac{1}{k}}{2\ln{k}} = \frac{k-\frac{1}{k}}{\ln{k}-\ln{(\frac{1}{k}})} \geq \sqrt{k \cdot \frac{1}{k}}
 \implies k-2\ln{k}-\frac{1}{k} \geq 0.
$$
Combine this with $\sqrt{x}-\sqrt{y} \geq 0$ and $y > 0$ to get 
$$
\dv{\psi{(x,y)}}{x} \geq 0.
$$
We can use the same way to show this inequality is correct when $y > x > 0$.
\endproof

\section{Acknowledgments} 

Gheorghe Craciun and Jiaxin Jin acknowledge support from the National Science Foundation under grants DMS--1412643 and DMS--1816238. Casian Pantea was partially supported by National Science Foundation grant DMS--1517577. Adrian Tudorascu acknowledges support from the National Science Foundation under grant DMS--1600272, and from the MSRI Research Professorship in the program ``Hamiltonian systems, from topology to applications through analysis'' from the Fall of 2018.


\begin{thebibliography}{99}

\bibitem{Anderson.2008aa}
D.F. Anderson, 
{\em Global asymptotic stability for a class of nonlinear chemical equations,} 
{SIAM J. Appl. Math}, 68:5, 2008. 

\bibitem{Anderson.2010aa}
D.F. Anderson, A. Shiu,
{\em The dynamics of weakly reversible population processes near facets,}
SIAM J. Appl. Math. 70:6, 2010.

\bibitem{Anderson.2011aa}
D. F. Anderson, 
{\em A proof of the Global Attractor Conjecture in the single linkage class case},
 SIAM J. Appl. Math., 71:4, 2011.

\bibitem{AMTU}{ A. Arnold, P. Markowich, G. Toscani, A. Unterreiter}, {\em On convex Sobolev inequalities and the rate of convergence to equilibrium for Fokker-Planck type equations}, Comm. Partial Differential Equations {\bf 26} (2001), pp. 43--100.

\bibitem{ChenLiWright} W.~Chen, C.~Li, and E.~Wright. 
\newblock {\em On A Nonlinear Parabolic System-Modeling Chemical Reactions In Rivers.} \newblock {Communications On Pure And Applied Analysis}, 4(4):889--899, 2005.

\bibitem{ChoulliKayser}{ M. Choulli, L. Kayser}, {\em Observations on Gaussian upper bounds for Neumann Heat Kernels},  Bulletin of the Australian Mathematical Society {\bf 92}, no. 3 (2015), pp. 429--439.

\bibitem{Craciun.2009aa}
G. Craciun, A. Dickenstein, A. Shiu, B. Sturmfels, 
{\em Toric Dynamical Systems, }
{Journal of Symbolic Computation} 44:11,  2009.

\bibitem{Craciun.2013aa}
G. Craciun, F. Nazarov, C. Pantea,
{\em Persistence and permanence of mass-action and power-law dynamical systems,}
SIAM J. Appl. Math. 73, 2013.

\bibitem{Craciun.gac} G.~Craciun. \newblock Toric Differential Inclusions and a Proof of the Global Attractor Conjecture. \newblock {arXiv:1501.02860}, 2016.

\bibitem{DF06}{ L. Desvillettes, K. Fellner}, {\em Exponential decay toward equilibrium via entropy methods for reaction-diffusion equations},J. Math. Anal. Appl., {\bf 319}  (2006), pp. 157--176.

\bibitem{DF08}{ L. Desvillettes, K. Fellner}, {\em Entropy Methods for Reaction-Diffusion Equations: Slowly Growing A-priori Bounds}, Rev. Mat. Iberoamericana, {\bf 24} (2008), pp. 407--431.

\bibitem{DF14}{ L. Desvillettes, K. Fellner}, {\em Exponential Convergence to Equilibrium for a Nonlinear Reaction-Diffusion Systems Arising in Reversible Chemistry}, System Modelling and Optimization, IFIP AICT, {\bf 443} (2014), pp.96--104.

\bibitem{DFPV07}{ L. Desvillettes, K. Fellner, M. Pierre, J. Vovelle}, {\em About Global Existence for Quadratic Systems of Reaction-Diffusion}, J. Adv. Nonlinear Stud. {\bf 7} (2007), pp. 491--511.

\bibitem{DFT2}{ L. Desvillettes, K. Fellner, B.Q. Tang}, {\em Trend to equilibrium for Reaction-Diffusion system arising from Complex Balanced Chemical Reaction Networks}, SIAM J. Math. Anal. {\bf 49}, no. 4 (2017), 2666--2709.

\bibitem{Feinberg.1972}
M. Feinberg.
{\em Complex balancing in general kinetic systems,}
Archive for Rational Mechanics and Analysis 49:3,
1972.

\bibitem{Fellner.2015aa} K.~Fellner, W.~Prager, and B.~Q.~Tang. \newblock {\em The entropy method for reaction-diffusion systems without detailed balance: first order chemical reaction networks.} \newblock {arXiv:1504.08221}, 2015.

\bibitem{FMS92}{ W.E. Fitzgibbon, J. Morgan, R. Sanders}, {\em Global existence and boundedness for a class of inhomogeneous semilinear parabolic systems}, Nonlin. Anal. {\bf 19}, no. 9 (1992), pp. 885--899. 

\bibitem{Gopalkrishnan.2014}
Manoj Gopalkrishnan,  Ezra Miller and Anne Shiu.
{\em A geometric approach to the global attractor conjecture} 
SIAM J. Appl. Dyn. Syst., 13:2, 2014.

\bibitem{Gopalkrishnan.2013}
Manoj Gopalkrishnan,  Ezra Miller and Anne Shiu.
{\em A projection argument for differential inclusions, with applications to persistence of mass-action kinetics}
SIGMA 9, 2013.

\bibitem{Horn.1972}
F. Horn, R. Jackson,
{\em General mass action kinetics, }
Archive for Rational Mechanics and Analysis 47,  1972.

\bibitem{Horn.1974}
F. Horn, 
{\em The dynamics of open reaction systems.} 
Mathematical aspects of chemical and biochemical problems and quantum chemistry, 
SIAM-AMS Proceedings Vol. VIII,  1974. 

\bibitem{LadSolUral}{ O.A. Ladyzenskaya, V.A. Solonnikov, N.N. Ural'ceva}, {\em Linear and quasilinear equations of parabolic type}, Trans. Math. Monographs, AMS {\bf 23} (1995).

\bibitem{MHM}{ A. Mielke, J. Haskovec, P. A. Markowich}, {\em On uniform decay of the entropy for reaction-diffusion systems}, J. Dynam. Differential Equations {\bf 27} (2015), pp. 897--928.

\bibitem{FatmaCA}{ F. Mohamed, C. Pantea, A. Tudorascu}, {\em  Chemical reaction-diffusion networks; convergence of the method of lines}, J. Math. Chem. {\bf 56}, no. 1 (2018), pp. 30--68.

\bibitem{Pantea.2012}
{ C. Pantea},
{\em On the persistence and global stability of mass-action systems}, 
SIAM J. Math. Anal. 44:3, 2012. 

\bibitem{Protter-Weinberger}{ M.H. Protter, H.F. Weinberger}, \newblock {\em Maximum Principles in Differential Equations}, Springer, 2nd Edition (1984).

\bibitem{Rothe.1984}
F. Rothe. 
{\em Global solutions of reaction-diffusion systems.} 
Lecture Notes in Mathematics. Vol 1072. Springer, 1984.

\bibitem{Siegel.2000}
D. Siegel, D. MacLean. 
{\em Global stability of complex balanced mechanisms}, 
J. Math. Chem. 27:1, 2000.

\bibitem{Sontag.2001}
E. Sontag. 
{\em Structure and stability of certain chemical networks and applications to the kinetic proofreading model of T-cell
receptor signal transduction},
IEEE Trans. Automat. Control 46, 2001.

\end{thebibliography}
\end{document}